\documentclass[11pt]{amsart}

\usepackage{graphicx,amssymb,epstopdf,subcaption,mathtools,tikz-cd,pb-diagram,thmtools, thm-restate,pinlabel,pifont,enumitem,stmaryrd}
\usepackage{pinlabel}
\usepackage[hmargin=1.55in, vmargin=1.30in]{geometry}
\usepackage[]{hyperref}
\usepackage{tikz}
\usetikzlibrary{matrix}

\usepackage[all]{xy}
\usepackage[mathscr]{euscript}

\newcommand{\nc}{\newcommand}
\nc{\dmo}{\DeclareMathOperator}
\nc{\nt}{\newtheorem}
\nc{\p}[1]{\bigskip\noindent\emph{#1.}}

\nt{theorem}{Theorem}[section]
\nt{corollary}[theorem]{Corollary}
\nt{lemma}[theorem]{Lemma}
\nt{defn}[theorem]{Definition}
\nt{criterion}[theorem]{Criterion}

\nt{question}[theorem]{Question}
\nt{problem}[question]{Problem}
\nt{conjecture}[question]{Conjecture}
\nt{proposition}[theorem]{Proposition}\nt*{claim}{Claim}

\nc{\M}{{\mathcal{M}_g^1}}
\nc{\X}{\mathcal{X}}
\nc{\C}{\mathcal{C}}
\nc{\T}{\mathcal{T}}
\nc{\W}{{\mathcal{W}_g^1}}
\nc{\J}{{\mathcal{K}_g^1}}
\nc{\Ch}{{\text{Ch}{_g^1}}}
\nc{\ch}{e}
\nc{\ii}{{\hat{i}}}
\nc{\I}{{\mathcal{I}_g^1}}

\nc{\cut}{\!\ssearrow\!}

\dmo{\Mod}{Mod}
\dmo{\SMod}{SMod}
\dmo{\SL}{SL}
\dmo{\PSp}{PSp}
\dmo{\PSL}{PSL}
\dmo{\Homeo}{Homeo}
\dmo{\im}{\mathrm{im}}
\dmo{\Aut}{Aut}
\dmo{\Sp}{\mathrm{Sp}({2g}, \Z)}

\nc{\g}{{g^{\alpha}}}
\nc{\Z}{\mathbb Z}
\nc{\N}{\mathcal N}
\nc{\R}{\mathbb R}
\nc{\F}{\mathcal F}

\nc{\ga}{\gamma}
\nc{\de}{\delta}
\nc{\ep}{\epsilon}

\nc{\flm}{\lambda_{2}}

\nc{\normalclosure}[1]{\ensuremath{\left \langle \left \langle #1 \right \rangle \right \rangle}}

\title{The normal closure of a homological genus 0 bounding pair map}
\author{Lei Chen}
\author{Weiyan Chen}

\address{Lei Chen \newline Morningside Center of Mathematics, Chinese Academy of Sciences \newline  Adademy of Mathemacis and Systems Science, Chinese Academy of Sciences\newline  Beijing, 100190, China \\  chenlei@amss.ac.cn }
\address{Weiyan Chen \newline Yau Mathematical Sciences Center, Tsinghua University\newline Beijing, 100084, China \\  chwy@tsinghua.edu.cn}

\begin{document}

\begin{abstract}
Justin Lanier and the authors recently determined the group normally generated by a single bounding pair map of genus $n$. We related this subgroup with the Chillingworth subgroup and the Casson--Morita's $d$ map. In this paper, we extend the results to the case when $n=0$. Let $\M$ be the mapping class group, $\Ch$ be the Chillingworth subgroup and $d$ be the Casson--Morita's $d$-map. We show that $Ker(d)=[\Ch,\M]$ and it is generated by a single homological genus 0 bounding pair map. We also construct an element $H_0\in \Ch$, and show that $\Ch$ is normally generated by this single element $H_0$.
\end{abstract}

\maketitle

\vspace*{-4ex}

\vspace*{0in}
\vspace{.15in}

\vspace{-.15in}

\section{Introduction}
%closed or punctured surface?
Johnson \cite{JohnsonLantern} proved that the Torelli group is normally generated by a bounding pair map of genus 1. In the authors' previous paper joint with Lanier \cite{Normal}, we 
studied the subgroup normally generated by a bounding pair map  of genus $n$ in the mapping class group $\M$ of a genus $g$ surface with 1 bondary component. Let $\ch:\I\to H_1(S_g^1;\Z)$ denote the Chillingworth homomorphism. 
Let $\Ch[2n]$ denote the subgroup of $\I$ consisting of $f$ such that $\ch(f)=0\pmod{2n}.$ 
% $\Ch[2n]$ denote the subgroup of $\I$ consisting of elements whose image under Chillingworth homomorphism is $0\pmod{2n}$.
% Let $\ch:\I\to H_1(S_g^1;\Z)$ denote the Chillingworth homomorphism and define 
%         $$\Ch[2n]:=\{f\in\I\ : \ \ch(f)=0\pmod{2n}\}.$$
        Let $d_{4n}: \Ch [2n]\to\Z/4n\Z$ denote the Casson--Morita's $d$ map modulo $4n$. We proved:
\begin{theorem}[Theorem 1.1 and 1.2 in \cite{Normal}]
\label{thm:old}
        When $1\le n\le g-2$, the normal subgroup  of $\M$ generated by  a genus $n$ bounding pair map $BP_n$ is
        \begin{equation}
        \label{eq:Wn}
             \langle\langle BP_n\rangle\rangle
            =[\Ch [2n],\M ]=\ker(d_{4n}).
        \end{equation}
\end{theorem}
In this paper, we consider an analog of Theorem \ref{thm:old} in the case when $n=0$. The equality (\ref{eq:Wn}) obviously fails when $n=0$ because any genus 0 bounding pair map is trivial. Instead, we consider a \emph{homological genus 0 bounding pair map} $B_0:=T_aT_b^{-1}$ introduced recently by Kosuge \cite[Theorem 22]{Chilling}. See Figure \ref{fig:B0} below.
\begin{figure}[h]
        \centering
                \vspace{-20pt}
\includegraphics[width=0.5\linewidth]{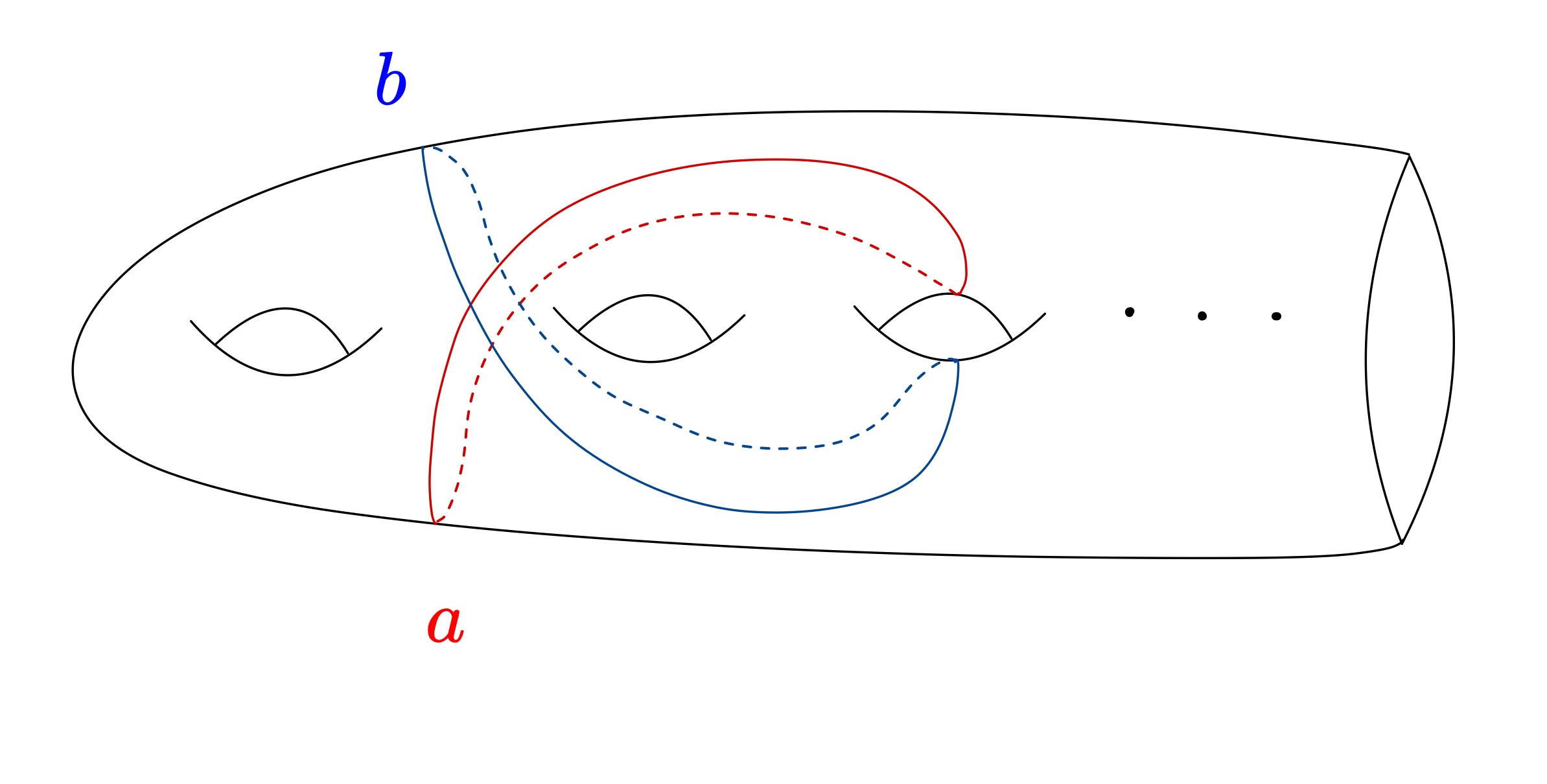}
        \vspace{-20pt}
        \caption{A homological genus 0 bounding pair map $B_0=T_aT_b^{-1}$.} 
                \vspace{-5pt}
        \label{fig:B0}
    \end{figure}

\noindent Intuitively, $B_0$ has homological genus 0 because it is the difference (and not the sum) of two genus 1 bounding pair maps. In Section 2, we will formally define the homological genus of a fake bounding pair map $T_aT_b^{-1}$ where $a$ and $b$ are homologous curves that can possibly intersect. Denote the Chillingworth subgroup by $\Ch:=\ker (\ch:\I\to H_1(S_g^1;\Z))$. Let $d:\Ch\to\Z$ denote the Casson-Morita's $d$ map restricted to $\Ch$. Kosuge introduced $B_0$ in order to show that $\ker(d)$ is normally generated by $B_0$ together with $[\J,\M]$ and a genus 1 separating twist \cite[Theorem B]{Chilling}. We will show that Kosuge's normal generating set only needs one element. In fact, we prove more:
\begin{theorem}
\label{main1}
 When $g\ge 6$, the normal subgroup of $\M$ generated by $B_0$  is
 \begin{equation}
         \label{eq:B0=}
            \langle\langle B_0\rangle\rangle=[\Ch ,\M ]=\ker(d). 
 \end{equation}
\end{theorem}

As a corollary, Theorem \ref{main1} implies that a genus 1 separating twist, which we denote by $T_1$, is a product of conjugates of $B_0$. However, our proof is indirect. We ask:
\begin{question} What is an explicit way to write a genus 1 separating twist $T_1$ as a product of conjugates of $B_0$?
\end{question}
We remark that any answer to this question can only work for genus 1 separating twists because any separating twist of  genus $n>1$ does not belong to $\langle\langle B_0\rangle\rangle$.

Our next theorem will show that $\Ch$ is normally generated by a single mapping class $H_0:=T_{a_1}T_{a_3}T_{a_2}^{-2}$ where $a_1,a_2,a_3$ are as in Figure \ref{fig:a1a2a3} below:
\begin{figure}[h]
\centering\includegraphics[width=0.6\linewidth]{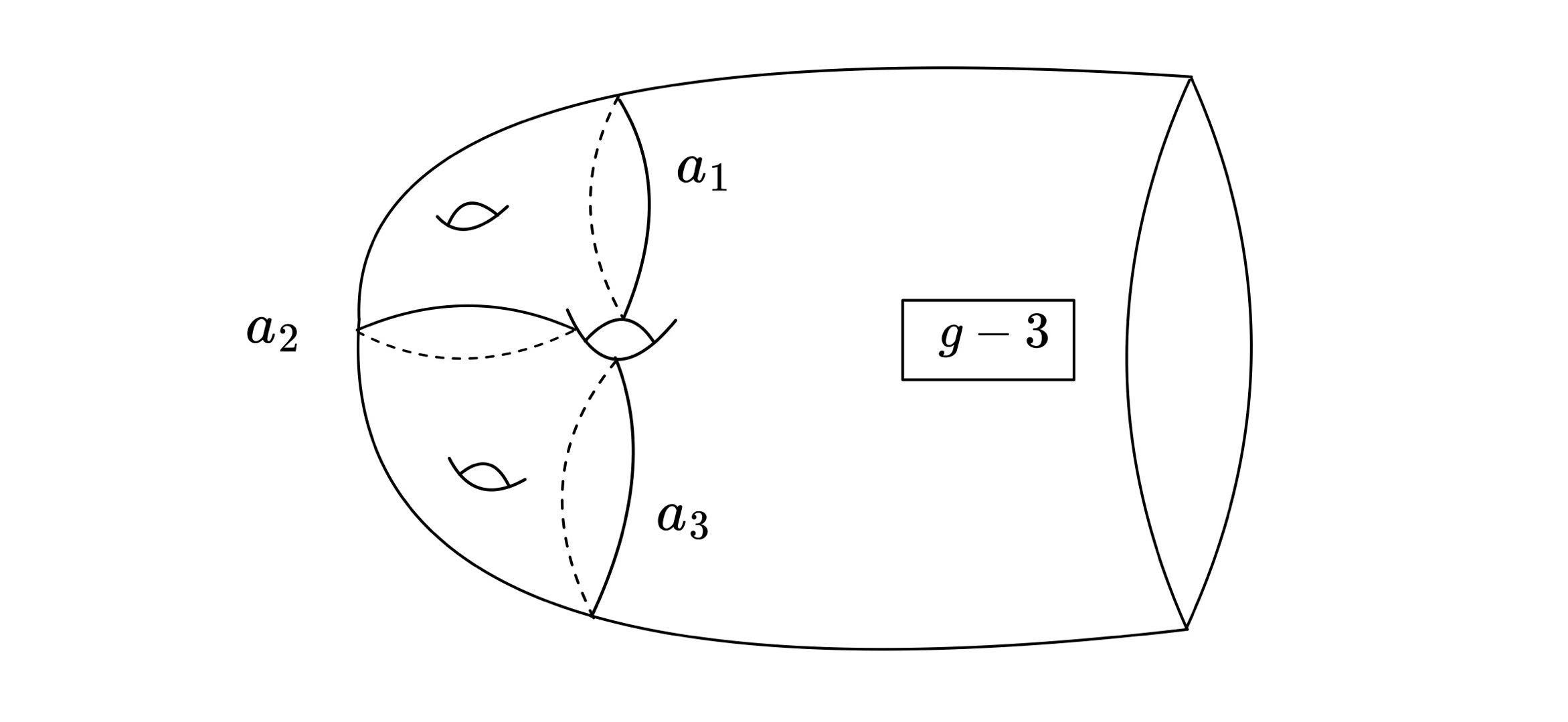}
    \caption{The number $g-3$ denotes the genus of the subsurface.}
    \label{fig:a1a2a3}
\end{figure}
\begin{theorem}\label{main2}
For $g\ge 6$, we have that
\[\langle \langle H_0\rangle \rangle = \Ch\]
\end{theorem}

Let $\W(0):=\langle\langle B_0\rangle\rangle$. We also prove the following homological results. 
\begin{theorem}
    \label{thm:homological results}
    For $g\ge 6$,  we have that
    \begin{enumerate}
        \item $H_1(\W(0);\Z)_{\M}=0,$\ \  \ \ \ \ \ \ \ \  $H^1(\W(0);\Z)^{\M}=0.$
        \item     $H_1(\Ch;\Z)_{\M}\cong \Z, \ \ \text{ }\ \ \ \ \ \ \ \ H^1(\Ch;\Z)^{\M}= \langle d/8\rangle\cong \Z.$\\
        $H^1(\Ch;\Z/m\Z)^{\M}= \langle d/8\rangle\cong \Z/m\Z$ for any $m\in\Z_{>0}$.
    \end{enumerate}
\end{theorem}
Kosuge \cite[Theorem C]{Chilling} computed $H_1(\Ch;\mathbb{Q})$ as $\M$-modules when $g\ge 6$. Our Theorem \ref{thm:homological results} provides the $\M$-invariant part over $\Z$ and $\Z/m\Z$.

\p{An important ingredient}
A key step in the proof of Theorem \ref{main1} is to prove that a curve complex $C_0^\alpha$, to be described below, is connected. In Section 2, we will define the \emph{homological genus} $g(a,b)$ of two homologous curves $a$ and $b$, not necessarily disjoint. For   a  nonseparating curve $\alpha$, we define a complex $C_0^\alpha$ to have:
\begin{itemize}
\item vertex: a curve $c$ such that $g(\alpha,c)=0$. 
\item edge: two vertices $c,d\in C_0^a$ are connected by an edge if and only if $T_cT_d^{-1}$ is conjugate to $B_0$ in $\M$.
\end{itemize}
In other words, there is an edge between $c,d\in  C_0^\alpha$ if and only if there is a mapping class $f$ such that $\{f(c),f(d)\}=\{a,b\}$ for $a,b$ in Figure \ref{fig:B0}. %The main tool to prove Theorem \ref{main1} is the following connectivity result.
\begin{theorem}\label{C_0}
Suppose that $g\ge 6$.
\begin{enumerate}
    \item The curve complex $C_0^\alpha$ is path-connected.
    \item The groups $\Ch$ and $\W(0)$ act on  $C_0^\alpha$  and both actions are transitive.
\end{enumerate}
\end{theorem}
\begin{corollary}[$\Ch$–orbits of nonseparating curves] 
\label{cor: ch orbit nonseparating}
Suppose that $g\ge 6$. For any two nonseparating, homologous curves $c$ and $d$,  the following conditions are equivalent:
\begin{itemize}
\item 
$c$ and $d$ are equivalent under $\Ch$, or equivalently, there exists an $f\in \Ch$ such that $T_c=fT_df^{-1}$.
\item 
$T_cT_d^{-1}\in \Ch$.
\item
the homological genus $g(c,d)$ is 0.
\end{itemize}
\end{corollary}
This corollary analogous to a theorem of Church \cite[Theorem 1.1]{Church}, which provides equivalent conditions for two nonseparating curves to be equivalent under the Johnson kernel $\J$. In the same paper, Church also gave equivalent conditions for two separating curves to be equivalent under $\J$. Hence, we ask:
\begin{question}
    What are the conditions, similar to those in Corollary \ref{cor: ch orbit nonseparating}, for two separating curves to be equivalent under $\Ch$?
\end{question}

Moreover, we do not know whether the hypothesis in Theorem \ref{main1} that $g\ge 6$ is necessary, although our proof requires this condition.
\begin{question} 
Can Theorem \ref{main1} be extended to surfaces with genus $g<6$?
\end{question}

On Page 255 of \cite{JohnsonConjugacy}, Johnson asked whether it is possible to define the Birman--Craggs homomorphisms algebraically without using 4-dimensional topology. We are in a similar situation. When we prove $\Ch/[\Ch,\M]\cong \Z$ (or even just that $\Ch\neq [\Ch,\M]$), we need to use the Casson-Morita's $d$-map, which involves $4$-dimensional topology. Is there an algebraic way to see that $\Ch/[\Ch,\M]\cong \Z$?

\p{Strategy and outline}
We use the lantern relations in various ways. %During the proof, we need to identify the curves we draw. 
In Section 2, we prove a rigidity result about two homologous curves intersecting at two points. This topological result helps us to recognize mapping classes that are conjugates of $B_0$. It then allows us to prove that $[\Ch,\M]$ contains $B_0$ and $T_1^2$. 
In Section 3, we use a homological argument about the cohomology of groups to obtain the key lemma showing that $T_1\in [\Ch,\M]$. This part of proof is indirect and  uses the results in \cite{Normal} in an unexpected way. As consequences, we prove the first equality of Theorem \ref{main1}, Theorem \ref{main2}, and Theorem \ref{thm:homological results}. 
In Section 4, we prove the connectivity of $C_0^\alpha$ and the last equality of Theorem \ref{main1}. The proof of the connectivity of $C_0^\alpha$ uses Putman's trick multiple times on multiple different curve complexes that we construct. We are curious if one can find a more direct proof.

\p{Acknowledgments} The authors would like to thank Dan Margalit and Justin Lanier for the conversations that set this work in motion. The second author is supported by the National Natural Science Foundation of China under the Young Scientists Fund No. 12101349.

\vspace*{4ex}

\section{The Chillingworth homomorphism and the Chillingworth subgroup}
Throughout this paper, we focus on the oriented surface $S_g^1$ of genus $g$ with 1 boundary component. Let $H$ denote $H_1(S_g^1;\Z)$. In this section, we establish some relations in $\Ch$. In 2.1, we define the homological genus of a fake bounding pair map. In 2.2, we prove a criterion for a mapping class to be conjugate to $B_0$. Then  we show that $B_0\in[\Ch,\M]$ in 2.3 and that $B_0\in[\W(0),\M]$ in 2.4. Finally, in 2.5, we show that $T_1^2\in [\Ch,\M]$ where again $T_1$ is a genus 1 separating twist. 

\subsection{Homological genus of a fake bounding pair map}
Two  homologous curves
bound a subsurface if they are disjoint. 
In this subsection, we will show that two homologous curves,  even if they intersect, can still have a well-defined ``homological genus" using the Chillingworth homomorphism.

Let us start by describing the Chillingworth homomorphism. We will follow Johnson's approach in \cite{JohnsonAbelian} and give an equivalent definition which is different from Chillingworth's original definition in \cite{chillingworth1972winding}. Consider the Johnson homomorphism $\tau: \I\to \wedge^3 H$ 
and a contraction map $C: \wedge^3 H \to H$ defined by 
\begin{equation}
    \label{eq: C def}
    C(x \wedge y \wedge z) =2\Big[(x\cdot y) z+ (y\cdot z) x+ (z\cdot x) y\Big]
\end{equation}
where $x\cdot y$ denotes the intersection pairing of $x$ and $y$. 
The Chillingworth homomorphism $\ch:\I\to H$ is defined by 
\begin{equation}
    \label{eq:ch=C tau}
    e(f)=C(\tau(f))\ \ \ \ \ \ \ \ \ \forall f\in \I.
\end{equation}
One can evaluate the Chillingworth homomorphism on a bounding pair map following a standard calculation as demonstrated in  \cite{primer}, Section 6.6.2. We summarize this result as the following lemma. 
\begin{lemma}
\label{lem: ch of BPk is 2k}
Suppose that $(a,b)$ is a bounding pair such that they bound a subsurface   $\Sigma$ of genus $n$. Let $\vec{a}$ denote the curve $a$ together with the orientation such that $\Sigma$ is on the right hand side of $\vec{a}$. Then $$\ch(T_aT_b^{-1})=2n[\vec{a}].$$
\end{lemma}

By a curve $a$ we mean an isotopy class of unoriented simple closed curves in $S_g^1$ (i.e. embedded circles). We will use $\vec{a}$ to denote  the curve $a$ with an orientation, and denote its homology class by $[\vec{a}]$. Two curves $a,b$ are said to be \emph{homologous} if they can be oriented such that $[\vec{a}]=[\vec{b}]$ in $H$. If $a$ and $b$ are homologous but not necessarily disjoint, then $T_aT_b^{-1}\in\I$ is called a \emph{fake bounding pair map} as in \cite[Section 6.2.5]{primer}. It turns out that a fake bounding pair $(a,b)$ can still have a well-defined ``genus", even if $a$ and $b$ intersect.  %Let $\alpha\in H$ be a primitive homology class.
\begin{proposition}[homological genus of a fake bounding pair]
\label{prop: homo genus}
Consider $S_g^1$ when $g\ge 3.$     If $a$ and $b$ are two homologous curves which can be oriented such that $[\vec{a}]=[\vec{b}]=\alpha$, then the Chillingworth homomorphism $e:\I\to H$ satisfies that
    $$\ch(T_aT_b^{-1})= 2\g(a,b)[\vec{a}]$$
    for some $\g(a,b)\in\Z$. We will call this integer $\g(a,b)$ the \emph{homological genus} of the fake bounding pair $(a,b)$ with respect to $\alpha$.
\end{proposition}

We will prove Proposition \ref{prop: homo genus} using Putman's result on the connectivity of the \emph{homology curve complex} $C(\alpha)$.
\begin{defn}[homology curve complex] We define $C(\alpha)$ as the following.
    \begin{itemize}
\item vertex: $b\in C(\alpha)$ if and only if $b$ is a curve homologous to $\alpha$;
\item edge: $b,c\in C(\alpha)$ form an edge if $i(b,c)=0$.
\end{itemize}
\end{defn}
\begin{theorem}[Putman, Theorem 1.9  in \cite{Putman}]\label{thm: Putman homo cur comp}
When $g\ge 3$, the homology curve complex $C(\alpha)$ is path-connected.
\end{theorem}

\begin{proof}[Proof of Proposition \ref{prop: homo genus}]
    By Theorem \ref{thm: Putman homo cur comp}, we can find a path in $C(\alpha)$ connecting $a$ and $b$:
    $$a_0, a_1,\cdots, a_m\qquad \text{where } a_0=a\text{ and } a_m=b \text{ and } i(a_j,a_{j+1})=0.$$
    Then we have that
    $$T_aT_b^{-1}= T_{a_0}T_{a_1}^{-1} T_{a_1}T_{a_2}^{-1}...T_{a_{m-1}}T_{a_m}^{-1}.$$
    By Lemma \ref{lem: ch of BPk is 2k}, if we orient $\vec{a}_i$ such that $a_i$ and $a_{i+1}$ bound a genus $k$ subsurface on the right side of $\vec{a}_i$, then we have that    $$\ch(T_{a_i}T_{a_{i+1}}^{-1})=2k[\vec{a}_i]=\pm 2k\alpha
    $$
Then we have that 
$$\ch(T_aT_b^{-1})= \sum_{i=0}^{m-1}\ch(T_{a_i}T_{a_{i+1}}^{-1})\in 2\Z \alpha.$$
\end{proof}

Here are some basic properties of $\g$ all of which follow easily from its definition.
\begin{proposition}
\label{prop:g properties}
For any homologous curves $a,b,c$, all of which can be oriented to represent the same homology class $\alpha,$ we have the followings:
\begin{enumerate}
    \item(Antisymmetry) $\g(b,a)=-\g(a,b)$.
    \item(Transitivity) $\g(a,c)=\g(a,b)+\g(b,c).$
    \item(Reversing the orientation) $g^{-\alpha}(a,b)=-\g(a,b)$.
    \item(Determining the sign) If $a,b$ are disjoint, bound a subsurface $\Sigma$ of genus $k$, and are oriented such that $[\vec a]=[\vec b]=\alpha$, then  
$$\g(a,b)=
\begin{cases}
    k &\text{if $\Sigma$ lies on the right side of $\vec a$
and on the left side of $\vec b$}\\
    -k &\text{if $\Sigma$ lies on the left side of $\vec a$
and on the right side of $\vec b$}
\end{cases}$$
    \end{enumerate}    
\end{proposition}

\begin{proposition}
\label{prop: g of b and Ch(b) is zero}
    For any $f\in\Ch$ and any curve $a$, we have that $\g(a,f(a))=0.$
\end{proposition}
\begin{proof}
    $\ch(T_aT_{f(a)}^{-1}) =\ch(T_afT_{a}^{-1}f^{-1}) = \ch(T_afT_{a}^{-1})+\ch(f^{-1})=0+0.$
\end{proof}

For example, the curves $b,c$ on the left side of Figure \ref{fig:Bmm} below satisfy $\g(b,c)=0$. This is because can find a third nonseparating curve $a$ such that $(a,b)$ and $(a,c)$  both bound two  subsurfaces of genus $1$ that are on the same side of $a$. By Proposition \ref{prop:g properties}, we can fix a homology class  $\alpha$ such that 
$$\g(b,c)=\g(a,c)-\g(a,b)=1-1=0.$$
As in the introduction, we will use $B_0$ to denote any element in $\M$ that  is conjugate  to $T_bT_c^{-1}$. More generally, the curves $b,c$ on the right side of Figure \ref{fig:Bmm} also  satisfy $\g(b,c)=0$. In this case, we will use $B_{m,m}$ to denote  any element in $\M$ that  is conjugate  to $T_bT_c^{-1}$. In particular, we have $B_0=B_{1,1}$. The fake bounding pair map $B_{m,m}$ is of homological genus 0 for any $m$. 
    \begin{figure}[h]
    \centering
    \vspace{-10pt}
    \includegraphics[width=1\linewidth]{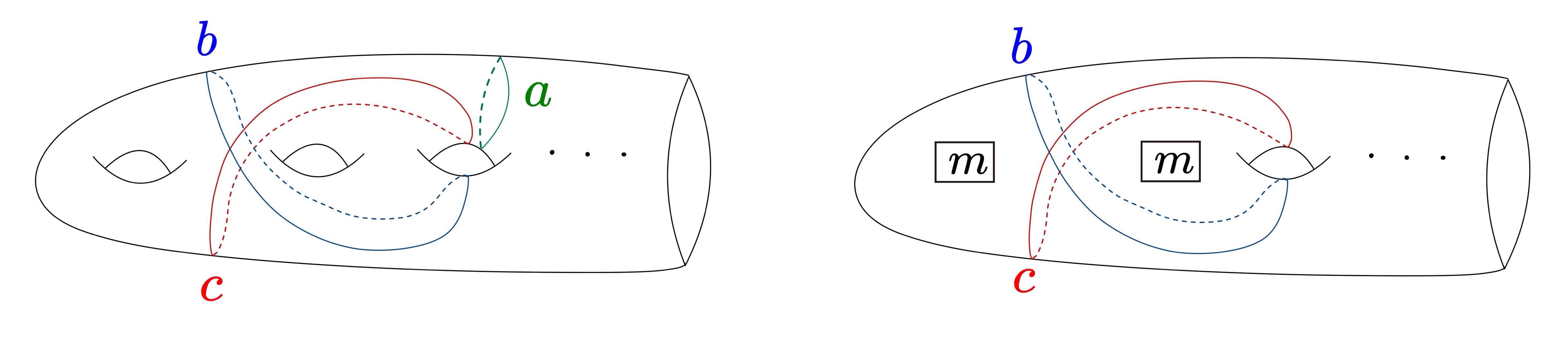}
        \vspace{-20pt}
    \caption{The boxed value $m$ denotes the genus of the subsurface.}
    \label{fig:Bmm}
\end{figure}

\vspace{5pt}

\subsection{Classification of homological genus 0 bounding pair maps}
In this subsection, we show that any homological genus 0 bounding pair maps $T_bT_c^{-1}$ such that $i(b,c)=2$ must be conjugate to either $B_0$ or $B_{m,m}$ as in Figure \ref{fig:Bmm} above.

\begin{lemma}[\textbf{Classification of homological genus 0 bounding pair maps}]
\label{lem: i=2}
    For any two curves $b$ and $c$, the mapping class $T_bT_c^{-1}$ is conjugate to $B_{m,m}$ for some positive integer $m$ if and only if the following conditions hold:
    \begin{enumerate}
        \item $b,c$ are nonseparating and homologous.
        \item $i(b,c)=2$.
        \item $\g(b,c)=0$.
    \end{enumerate}
    Furthermore, $T_bT_c^{-1}$ is conjugate to $B_{1,1}=B_0$ if $b\cup c$ are contained in a subsurface of genus $\le 4$ with 1 boundary component.
\end{lemma}
\begin{proof}
It is clear from Figure \ref{fig:Bmm} that the three conditions are necessary. We need to show that they are also sufficient. Choose orientations on $b$ and $c$ such that $[\vec{b}]=[\vec{c}]$ in $H$. We will fix the orientations throughout the proof, and for simplicity still use $b,c$ to denote the oriented curves $\vec{b},\vec{c}$.  

Let $N$ denote a small neighborhood of $b\cup c$. The closure of $N$ forms a lantern, i.e. a genus 0 subsurface with 4 boundary components. Let $A,B,C,D$ denote the four oriented arcs in ${b}$ and ${c}$ connecting their two intersection points such that ${b}=AB$ and ${c}=CD$. Here we follow the convention that $AB$ means the curve going along $A$ first  and then  $B$. See Figure \ref{fig:lantern bc} below.  The four boundary components are isotopic to $AC^{-1}, BC, DB^{-1}, AD$, all of which are non-contractible curves, because otherwise $b$ and $c$ could be isotoped to be disjoint.

\begin{figure}[h]
    \centering
    \includegraphics[width=1\linewidth]{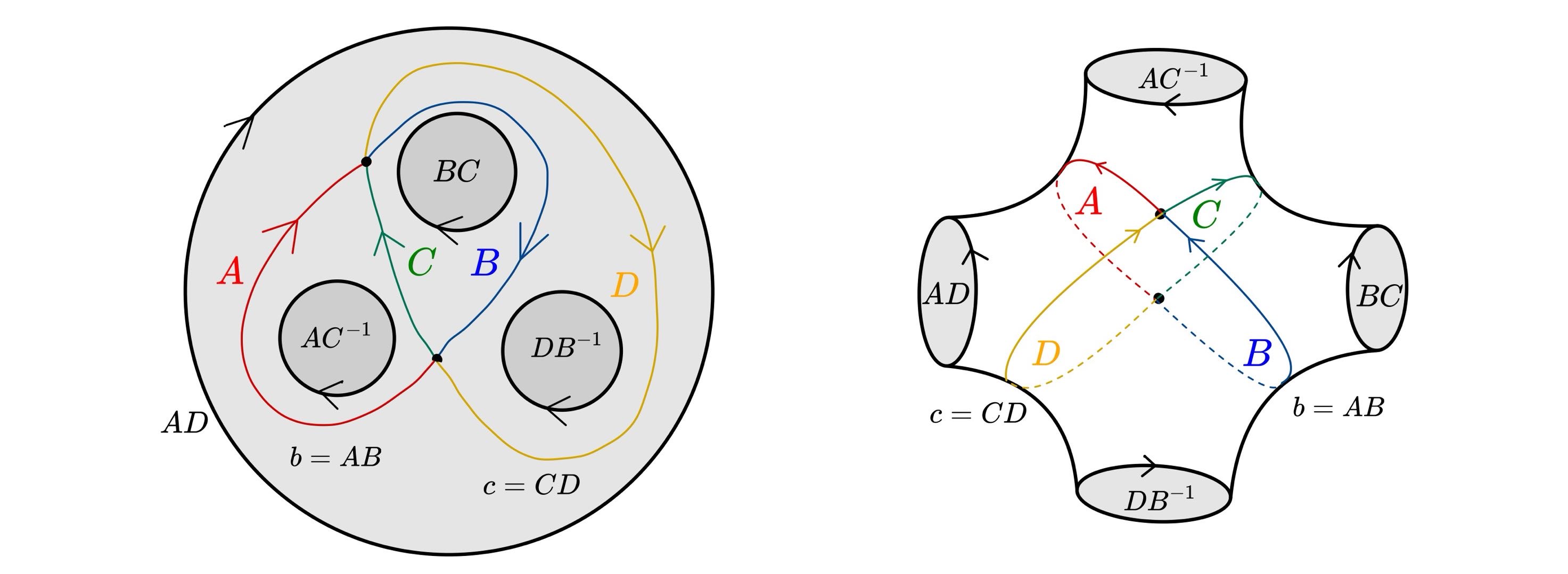}
    \caption{Left:  A small neighborhood of $b\cup c$ forms a lantern. $A,B,C,D$ denote the four arcs such that $b=AB$ and $c=CD$ with orientations. Right: A different picture of the same lantern $N$.}
    \label{fig:lantern bc}
\end{figure}

    \begin{figure}[h]
        \centering
        \includegraphics[width=1\linewidth]{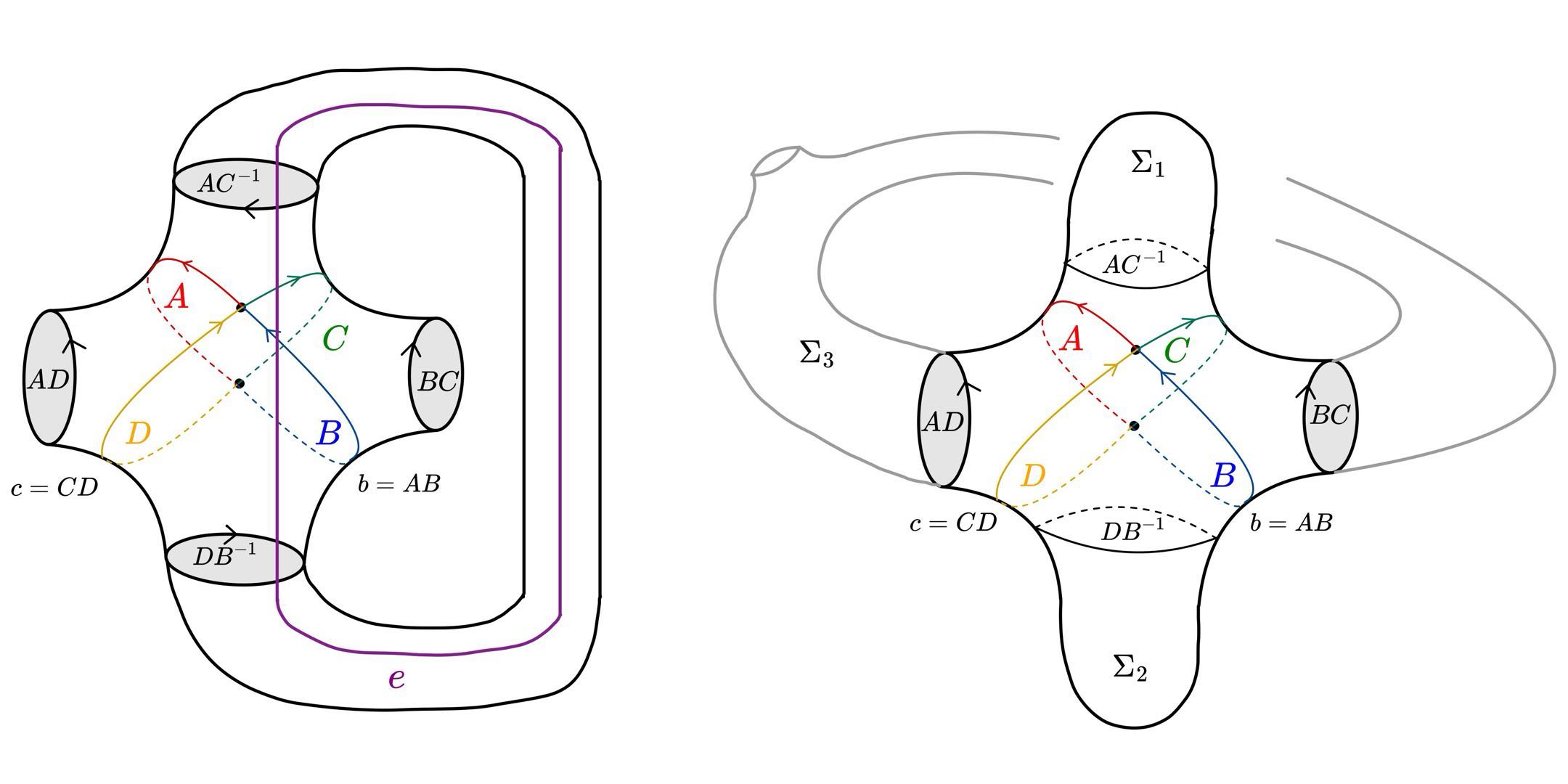}
    \caption{Left: %If $AC^{-1}$ and $DB^{-1}$ are nonseparating and form a bounding pair, then we can find an oriented curve (
    the purple curve $e$ intersects $AC^{-1}$ and $DB^{-1}$ once but with opposite signs. This is cannot happen. Right: the three subsurfaces $\Sigma_1, \Sigma_2, \Sigma_3$ outside of $N$.}
        \label{fig:recognize Bmm}
    \end{figure}

The rest of the proof will proceed in the following steps.
\vspace{5pt}

\textbf{Step 1: $\mathbf{AC^{-1}}$ and $\mathbf{DB^{-1}}$ are separating curves, and $\mathbf{(BC,AD)}$ forms a bounding pair.} 
We have the following equalities in $H=H_1(S_g^1;\Z)$:
    \begin{align*}
    [AC^{-1}] &= [ABB^{-1}C^{-1}] = [AB]+[B^{-1}C^{-1}] \\
    &= [CD]+[B^{-1}C^{-1}] &\text{since $[AB]=[b]=[c]=[CD]$}\\
    &= [DB^{-1}].
    \end{align*}
    If $AC^{-1}$ is nonseparating, then the two homologous curves $AC^{-1}$ and $DB^{-1}$ form a bounding pair, with appropriate orientations. Then we can find an oriented  curve $e$ depicted on the left hand side of Figure \ref{fig:recognize Bmm} such that the following algebraic intersection numbers are equal:
    $$\hat{i}(e,AC^{-1}) = -\hat{i}(e,DB^{-1})=1. $$
    This contradicts with our previous calculation that $[AC^{-1}]=[DB^{-1}]$. Therefore, $AC^{-1}$ must be separating, and so must be $DB^{-1}$.    As a consequence, $BC$ has to be nonseparating because
    $$0\ne [b]= [AB]=[AC^{-1}]+[BC] = 0+[BC].$$
    By a similar argument, $[b]=[AD]$. Hence $(AD,BC)$ forms a bounding pair. In summary,  the complement $S_g^1\setminus N$ consists of three components. Among them, let $\Sigma_1$ and $\Sigma_2$ denote the subsurfaces bounded by the separating curves $AC^{-1}$ and $DB^{-1}$, respectively. Let $\Sigma_3$ denote the subsurface bounded by the bounding pair $(AD,BC)$. See the right side of Figure \ref{fig:recognize Bmm}. 

\vspace{5pt}

\textbf{Step 2: the boundary component of $S_g^1$ lies in $\Sigma_3$.}
First, we show that the boundary component of $S_g^1$ cannot lie on $\Sigma_1$. Let us assume the contrary and reach a contradiction. Then in this case, the subsurface bounded by the bounding pair $(AD,b)$ has positive genus (since it is the union of $\Sigma_2$ and a pair of pants) and is on the right side of $AD$ and on the left side of $b$. Therefore, by Proposition \ref{prop:g properties}, if we fix $\alpha:=[AD]=[b]$, then we have that
    $$\g(AD,b)=\text{genus of }\Sigma_2>0.$$
    In contrast, the subsurface bounded by the bounding pair $(AD,c)$ is on the left side of $AD$ and on the right side of $c$. Therefore, for the same choice $\alpha=[AD]=[c]$, we have by Proposition \ref{prop:g properties} that
    $$\g(AD,c)\leq 0.$$
    However, by Proposition \ref{prop:g properties} and our assumption that $\g(b,c)=0$, we have that 
    $$\g(AD,b) = \g(AD,b)+\g(b,c) =\g(AD,c).$$
    This is a contradiction. Similarly, the boundary component cannot lie on $\Sigma_2$, either. Therefore, the boundary component must lie on $\Sigma_3$, as we depicted on the right side  of Figure \ref{fig:recognize Bmm}. 
\vspace{5pt }

\textbf{Step 3: $\mathbf{\Sigma_1}$ and $\mathbf{\Sigma_2}$ have the same genus.} 
    Let $m,m'$ denote the genera of the subsurfaces $\Sigma_1$ and $\Sigma_2$, respectively. Again we fix $\alpha:=[b]=[c]=[AD]$. Then by Proposition \ref{prop:g properties}, we have that 
    % $\g(AD,b)=m'$ and $\g(AD,c)=m$. Therefore, we conclude that 
    $$m'=\g(AD,b)=\g(AD,b)+\g(b,c)=\g(AD,c)=m.$$
    % $$0=\g(b,c)=\g(AD,c)-\g(AD,b)= m-m'.$$
    Hence, the subsurface $N\cup \Sigma_1\cup\Sigma_2$ is  of genus $2m$ with 2 boundary components, as depicted on the right side of Figure \ref{fig:recognize Bmm}, which is the same picture as the right side of \ref{fig:Bmm}, up to a homeomorphism. Hence, $T_bT_c^{-1}$ is conjugate to $B_{m,m}$.
\vspace{5 pt}

    \textbf{Step 4: if $\mathbf{b,c}$ are contained in a subsurface $\Sigma$ of genus $\mathbf{\le 4}$ with 1 boundary component, then $\mathbf{m=1}$.}  Suppose that $m\ge 2$. Notice that both $AC^{-1}$ and $DB^{-1}$ are in the lantern $N$ which is a small neighborhood of $b\cup c$, and therefore are also in the subsurface $\Sigma$. Let $d$ denote the boundary of $\Sigma$. Then $d$ is a separating curve in the large surface $S_g^1$. Since $d$ and $N$ are disjoint, $d$ must be in one of the subsurfaces $\Sigma_i$ for $i=1,2,3$. If $d$ is in $\Sigma_1$, then $\Sigma$ is entirely contained in $\Sigma_1$ because we proved previously that the boundary of $S_g^1$ is in $\Sigma_3$. This contradicts our assumption that $b,c$ are in $\Sigma$. Similarly, $d$ cannot be in $\Sigma_2$. Hence, $d$ must be in $\Sigma_3$. In this case, $\Sigma$ contains the subsurface  $N\cup \Sigma_1\cup\Sigma_2$. However, $N\cup \Sigma_1\cup\Sigma_2$ is a subsurface of genus $2m\ge 4$ with 2 boundary components, while $\Sigma$ is of genus $\leq4$ with 1 boundary component. It is impossible that $\Sigma$ contains $N\cup \Sigma_1\cup\Sigma_2$. 
    
    Hence, we conclude that $m=1$ and that $T_bT_c^{-1}$ is conjugate to $B_{1,1}=B_0$. 
\end{proof}

\subsection{\boldmath The proof of that $B_0\in [\Ch,\M]$}
In this subsection, we prove that $B_0$ is an element in $[\Ch,\M]$.  Let $a_i$ for $i=1,2,3$ and $d_1$ be the curves in Figure \ref{fig:a1a2a3d1} below. Define $H_0:=T_{a_1}T_{a_3}T_{a_2}^{-2}$
and $d_3:=H_0(d_1)$. 

\begin{figure}[h]
    \centering
    \includegraphics[width=0.6\linewidth]{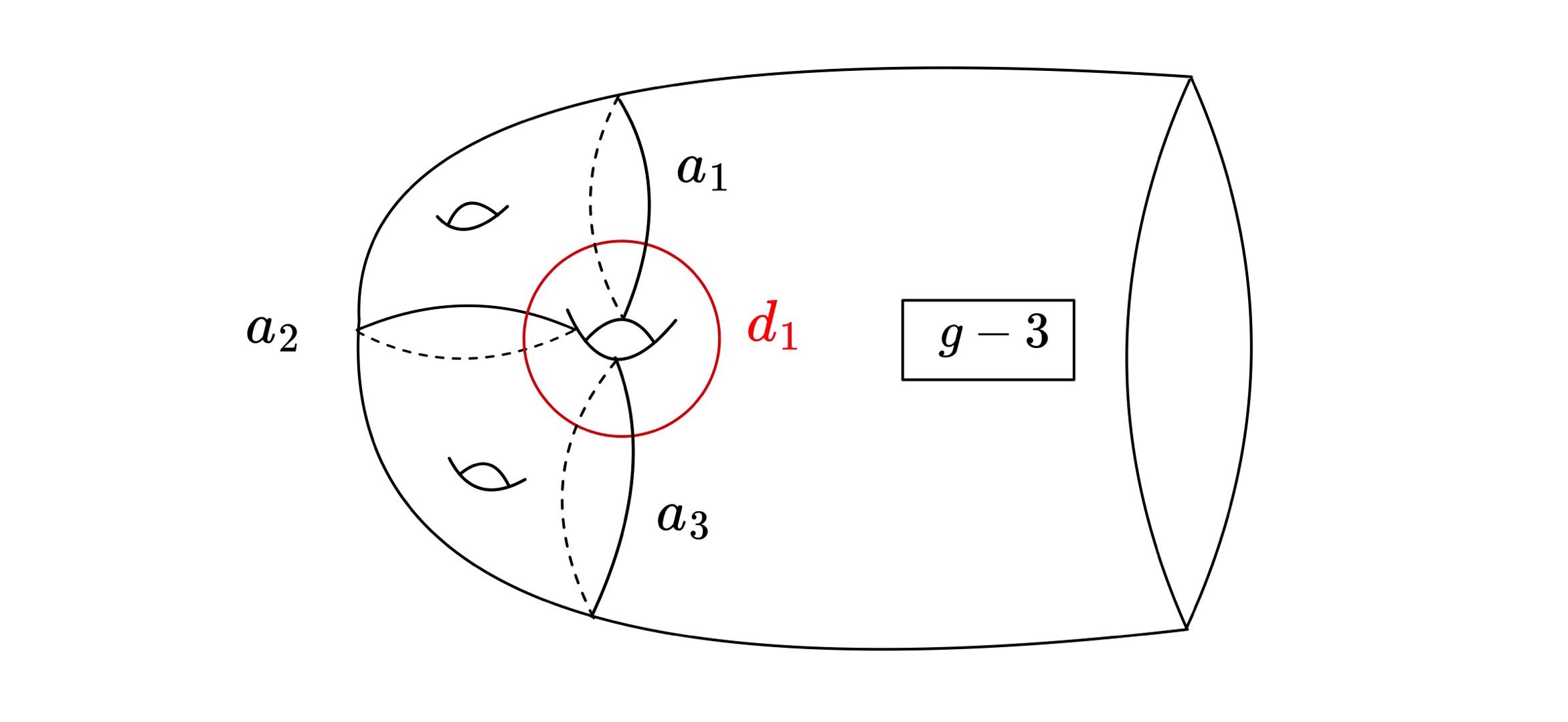}
    \vspace{-10pt}
    \caption{The number in the box denotes the genus of the subsurface.}
    \label{fig:a1a2a3d1}
\end{figure}

\begin{proposition}
    \label{2.2}
We have the following results:
\begin{enumerate}
    \item $H_0\in\Ch$.
    \item $[T_{d_1}, H_0]$ is conjugate to $B_0$ in $\M$.
    % \item $T_{d_1}T_{d_3}^{-1}$ is conjugate to $B_0$ in $\M$.
    \item $\W(0):=\langle\langle B_0\rangle\rangle\le [\Ch,\M]$.
\end{enumerate}

\end{proposition}
\begin{proof}
To prove (1), we compute the Chillingworth homomorphism: 
\begin{align*}
    \ch(H_0)&=\ch(T_{a_1}T_{a_3}T_{a_2}^{-2})= \ch\Big((T_{a_1}T_{a_2}^{-1})(T_{a_2}T_{a_3}^{-1})^{-1}\Big)\\
    &= \ch(T_{a_1}T_{a_2}^{-1})-\ch(T_{a_2}T_{a_3}^{-1})\\
    &= 2 [a_1]-2[a_2]=0
\end{align*}
The last line follows from Lemma \ref{lem: ch of BPk is 2k} where we orient $a_1$ such that the subsurface bounded by $(a_1,a_2)$ is on the right side of $a_1$ and we orient $a_2$ such that $[a_1]=[a_2]$, and hence the subsurface bounded by $(a_2,a_3)$ is on the right side of $a_2$.

To prove (2), we first observe that $$[T_{d_1},H_0]=T_{d_1}T_{H_0(d_1)}^{-1} =T_{d_1}T_{d_3}^{-1} \in [\Ch,\M].$$
Let $d_2:=T_{a_1}T_{a_2}^{-1}(d_1)$. As in the proof of \cite[Theorem 4.5]{Normal}, we can see that $i(d_1,d_2)=0$ and that $(d_1,d_2)$ form a genus 1 bounding pair. By the same argument, $(d_2,d_3)$ form a genus 1 bounding pair.

Moreover, we claim that $i(d_1,d_3)=2$. First of all, since each of $a_2$ and $a_3$ intersects $d_1$ at one point, we have that $i(d_1,d_3)=i(d_1,T_{a_3}T_{a_2}^{-1}(d_2))\le 2$.  Hence, $i(d_1,d_3)$ is either 0 or 2. It cannot be the case that $i(d_1,d_3)=0$ because if so, then $d_1,d_3$ form a bounding pair. By Proposition \ref{prop: g of b and Ch(b) is zero}, we have that $\g(d_1,d_3)=\g(d_1,H_0(d_1))=0$, which implies that $d_1$ and $d_3$ are isotopic. This is impossible since one can easily find a curve $e$ such that $i(d_1,e)\ne i(d_3,e)$. Hence, by Lemma \ref{lem: i=2}, the element $T_{d_1}T_{d_3}^{-1}$ is conjugate to $B_0$. 

To prove (3), we simply observe that by (2), $B_0$ is conjugate to $[T_{d_1}, H_0]\in [\Ch,\M]$.
\end{proof}
% As a corollary, we conclude that 
% \begin{corollary}
% \label{cor: B_0  in [Ch,M]}
% $\W(0):=\langle\langle B_0\rangle\rangle\le [\Ch,\M]$.
% \end{corollary}

% \begin{corollary}
% Let $\W(0):=\langle\langle B_0\rangle\rangle$. Then we have that 
% $$\W(0)=[\M,\W(0)]=[\M,[\M,\W(0)]].$$
% \end{corollary}
% \begin{proof}
% It suffices to prove that $\W(0)\le[\M,[\M,\W(0)]]$ since the other inclusions are clear. We have that 
% \begin{align*}
%     T_{d_1}T_{d_3}^{-1}&=[T_{d_1},h]\\
%     &=[T_{d_1}, (T_{a_1}T_{a_2}^{-1})(T_{a_2}T_{a_3}^{-1})^{-1}]\\
%     &=[T_{d_1}, (T_{a_1}T_{a_2}^{-1})f(T_{a_1}T_{a_2}^{-1})^{-1}f^{-1}]\\
%     &=[T_{d_1}, [f,T_{a_1}T_{a_2}^{-1}]]\in [\M,[\M,\W(0)]]
% \end{align*}
% \end{proof}

% \begin{corollary}
%     $H_1(\W(0))_\M=0$.
% \end{corollary}
We have the following generating set of $\Ch$. %%%here!!!
\begin{lemma}\label{normalch}
The Chillingworth group $\Ch$ is normally generated by $T_1,T_2,B_0$, and also by $T_1,T_2,H_0$. 
\end{lemma}
\begin{proof}
We already know that $T_1, T_2$ normally generate $\J=\ker \tau$. Moreover, $\tau(B_0)$ and its $\Sp$-image generate the image $\tau(\Ch)=U=\ker C$. The first claim thus follows. By (2) of Proposition  \ref{2.2}, we know that $\langle \langle H_0\rangle \rangle$ contains $B_0$. The second claim now follows from the first.
\end{proof}

% \begin{corollary}
%     $\W(0)=[\Ch,\M]$.
% \end{corollary}
% \begin{proof}
%     Suffice check that $[T_1,\M]$ and $[T_1,\M]$ are in $\W(0)$. \wc{to be done}
% \end{proof}

\subsection{\boldmath The proof of that $\W(0)= [\W(0),\M]$}
In this subsection, we will use lantern relations to prove that $\W(0)= [\W(0),\M]$.
\begin{lemma}
\label{2.3}
For any two curves $c$ and $d$ such that $T_dT_c^{-1}$ is conjugate to $B_0$, there exists a curve $a$ such that for $b:=T_dT_c^{-1}(a)$, the mapping class $T_aT_b^{-1}$ is also conjugate to $B_0$. As a consequence, we have that 
\[
     \W(0)=[\W(0),\M].\]
\end{lemma}
\begin{proof}
By Lemma \ref{lem: i=2}, if $T_dT_c^{-1}$ is conjugate to $B_0$, then the curves $c$ and $d$ can be conjugated by a mapping class into a configuration as in Figure \ref{fig:abcd} below. 

    \begin{figure}[h]
        \centering        \includegraphics[width=0.6\linewidth]{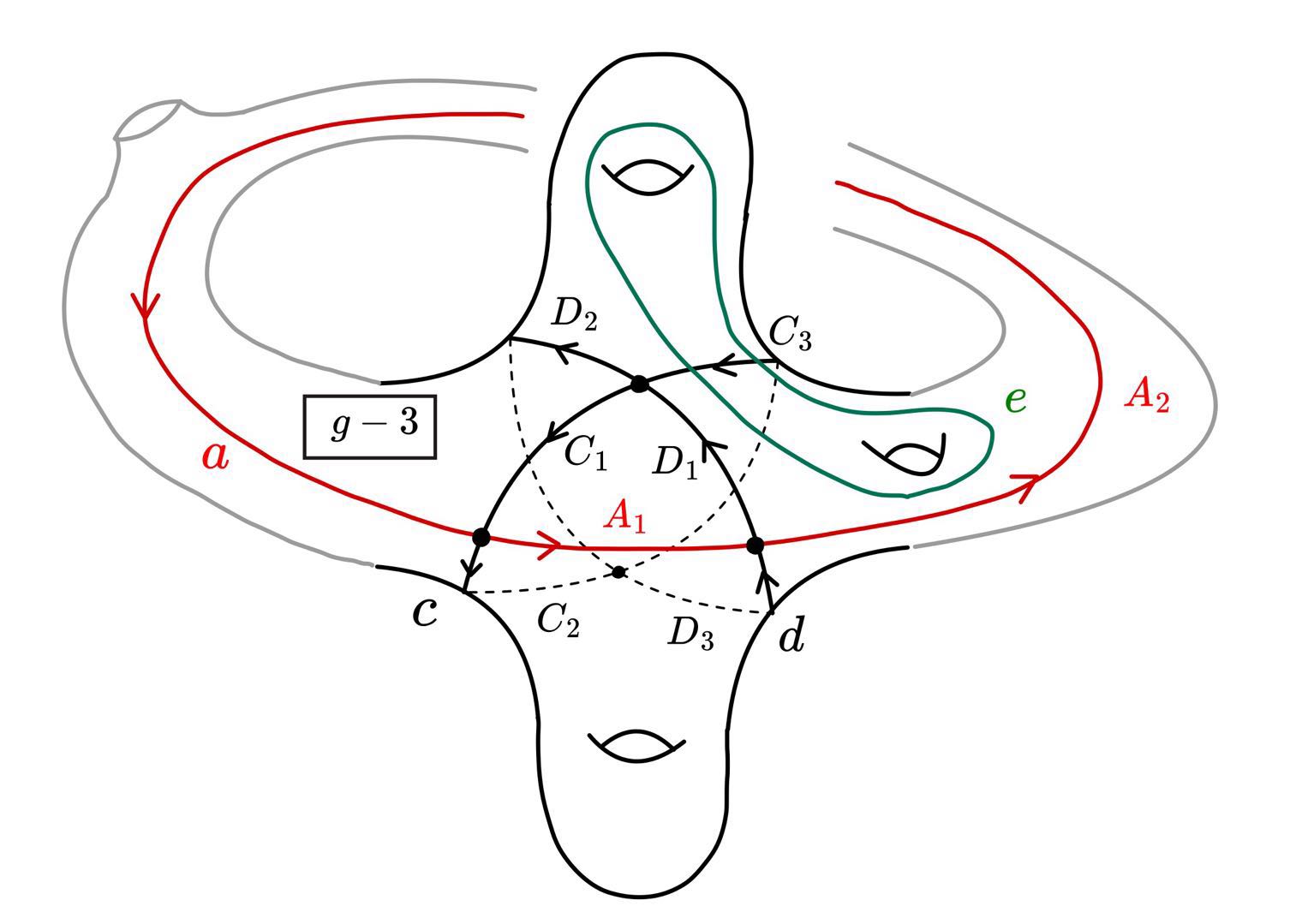}
        \caption{ }
        \label{fig:abcd}
    \end{figure}

We fix orientations on $c$ and $d$ as in the figure and let $a$ be the red curve. Let $b:=T_dT_c^{-1}(a)$. We will show that $T_aT_b^{-1}$ is conjugate to $B_0$, and consequently that 
$$T_aT_b^{-1}=[T_a, T_dT_c^{-1}]\in [\W(0),\M] \quad\Rightarrow\quad B_0\in [\W(0),\M].$$
By Lemma \ref{lem: i=2}, it suffices to show that $i(a,b)=2$.

The three curves $a,c,d$ intersect at three points, which divide $a,c,d$ into the following oriented arcs: 
$$a=A_1A_2,\ \ c=C_1C_2C_3,\ \ d=D_1D_2D_3.$$
We follow the convention that $A_1A_2$ means the curve going along $A_1$ first  and then $A_2$. We have 
$$i(a,b)=i(a,T_dT_c^{-1}(a)) = i(T_d^{-1}(a),T_c^{-1}(a)).$$
We can represent $T_d^{-1}(a),T_c^{-1}(a)$ as compositions of the following arcs: 
$$T^{-1}_d(a) = A_1D_3^{-1}D_2^{-1}D_1^{-1}A_2,\qquad T_c^{-1}(a) = A_2C_2C_3C_1A_2.$$
We calculate the number of intersection points as:
$$T_d^{-1}(a)\cap T_c^{-1}(a) = (a\cup d)\cap (a\cup c)= (a\cap c)\cup (d\cap a)\cup (d\cap c)= 4 \text{ points}.$$
We claim that exactly 2 of these 4 intersection points form a part of a bigon, and consequently $i(T_d^{-1}(a),T_c^{-1}(a))=2$. Indeed, it is straightforward to check that $T_d^{-1}(a)$ and $T_c^{-1}(a)$ form a bigon $D_1C_1A_1$. Hence $i(T_d^{-1}(a),T_c^{-1}(a))\le 4-2=2.$ We next claim that  $i(T_d^{-1}(a),T_c^{-1}(a))\ge 2.$

We will prove this statement by contradiction. The algebraic intersection numbers satisfy 
$$\ii(T_d^{-1}(a),T_c^{-1}(a))=\ii(a,T_dT_c^{-1}(a))=\ii(a,a)=0$$
because $T_dT_c^{-1}\in \I$. 
% of $a,b$ is $0$, we then know that the algebraic intersection number of $T_d^{-1}(a),T_c^{-1}(a)$ is also $0$. 
If the geometric intersection number  $i(T_d^{-1}(a),T_c^{-1}(a))< 2$, then $i(T_d^{-1}(a),T_c^{-1}(a))= 0$, or equivalently $T_d^{-1}(a),T_c^{-1}(a)$ are disjoint. Since $T_d^{-1}(a),T_c^{-1}(a)$ are also homologous, they form a genus 0 bounding pair and hence $T_d^{-1}(a)=T_c^{-1}(a)$. 

However, the curve $e$ depicted as the green curve in Figure \ref{fig:abcd} intersects $c$ at two points, but is disjoint from $a$ and $d$. Thus we have that $i(e,T_d^{-1}(a))=0$ but by \cite[Proposition 3.4]{primer}, we have
\[
|i(e,T_c^{-1}(a))-i(c,e)i(c,a)|\le i(e,a)=0,
\]
which implies 
\[
i(e,T_c^{-1}(a))=i(c,e)i(c,a)=2.
\]
This is a contradiction. The proof is complete.
\end{proof}

\subsection{\boldmath The proof of that $T_1^2\in [\Ch,\M]$}

Again, let $T_1$ denote a Dehn twist of a separating curve of genus $1$, which is unique up to conjugation in $\M$. In this subsection, we will prove the following.
\begin{lemma}\label{imd}
When $g\geq 3$, we have that 
$T_1^2\in [\Ch,\M]$ and that $d(T_2)=d(H_0)$.
\end{lemma}
\begin{proof}
By the change of coordinate principle, all the Dehn twists about a genus $i$ separating curve are conjugate to each other in $\M$. Let $T_i$ be an the Dehn twist about a genus $i$ separating curve, which is a well-defined element in $\Ch/[\Ch,\M]$. 
%Firstly, we have that $Ch_g^1$ is normally generated by $T_1,T_2,H_0$ by Lemma \ref{normalch}. All the following equations are equations in $\Ch/[\Ch,\M]$.
We will show that $T_1^2\equiv0$ in $\Ch/[\Ch,\M]$. In the proof below we will use $\equiv$ to denote an equality in $\Ch/[\Ch,\M]$.

In general, if two curves have geometric intersection number 2, then a small neighborhood of their union forms a lantern. Consider the lantern formed by the curves $b$ and $c$ in Figure \ref{fig:lantern for T_1^2} below. Let $b'$ denote the third diagonal curve in the lantern, which is also a genus 2 separating curve just like $b$ and $c$. The lantern relation gives that 
\begin{equation}
    \label{eq: lantern 1}
    T_bT_{b'}T_c \equiv T_1^3T_3 \quad \Rightarrow \quad T_2^3\equiv T_1^3T_3.
\end{equation}
Notice that elements in $\Ch/[\Ch,\M]$ only depends on its conjugacy class under $\M$ and that the order of products does not matter since $\Ch/[\Ch,\M]$ is abelian. 

\begin{figure}[h]
    \centering
    \includegraphics[width=0.35\linewidth]{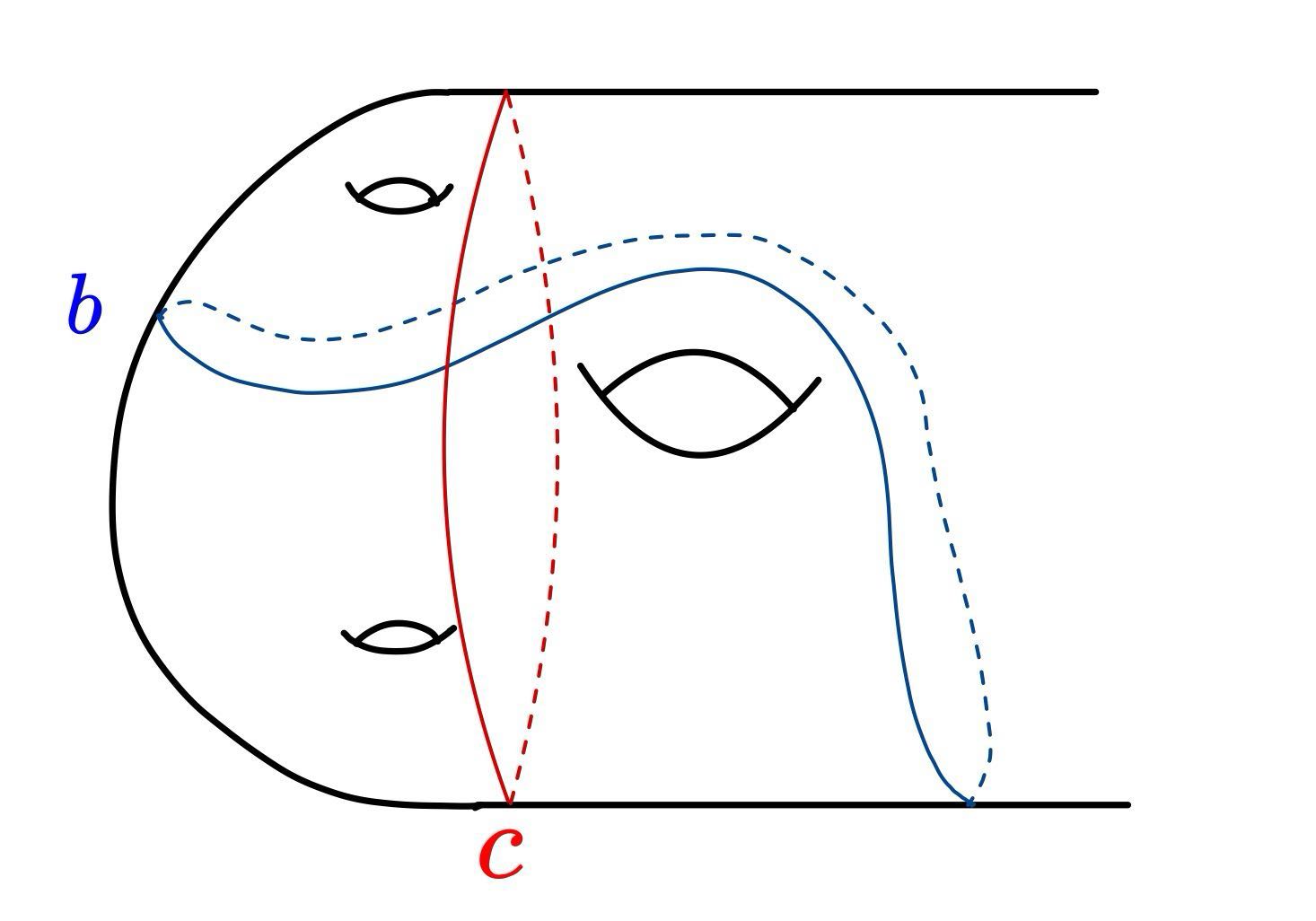}
        \caption{The lantern formed by $b$ and $c$.}
            \label{fig:lantern for T_1^2}
\end{figure}

Consider the lantern formed by the curves $a_1$ and $b$ depicted on the left side of Figure \ref{fig:two lanterns} below. Let $a_1'$ denote the third diagonal curve in this lantern. The lantern relation gives that 
\begin{equation*}
    T_{a_1}T_{a_1'}T_{b} \equiv T_{a_2} T_{a_3} T_1 T_3\quad \Rightarrow \quad (T_{a_1}T_{a_3}T_{a_2}^{-2})(T_{a_2}T_{a_1'}T_{a_3}^{-2}) \equiv T_1T_3T_2^{-1}
\end{equation*}
Moreover, the pair $(a_1',a_3)$ forms a genus 1 bounding pair and  $(a_1',a_2)$ forms a genus 2 bounding pair as we depict in the left side of Figure \ref{fig:two lanterns}. Hence, the triple of curves $(a_1,a_2,a_3)$ is conjugate to $(a_1',a_3,a_2)$ in $\M$, which implies that
$$H_0:=T_{a_1}T_{a_3}T_{a_2}^{-2}\equiv T_{a_2}T_{a_1'}T_{a_3}^{-2}.$$
Hence, the lantern relation above reduces to that 
\begin{equation}
    \label{eq: lantern 2}
    H_0^2 \equiv T_1T_3T_2^{-1}.
\end{equation}

\begin{figure}[h]
    \centering
    \includegraphics[width=0.8\linewidth]{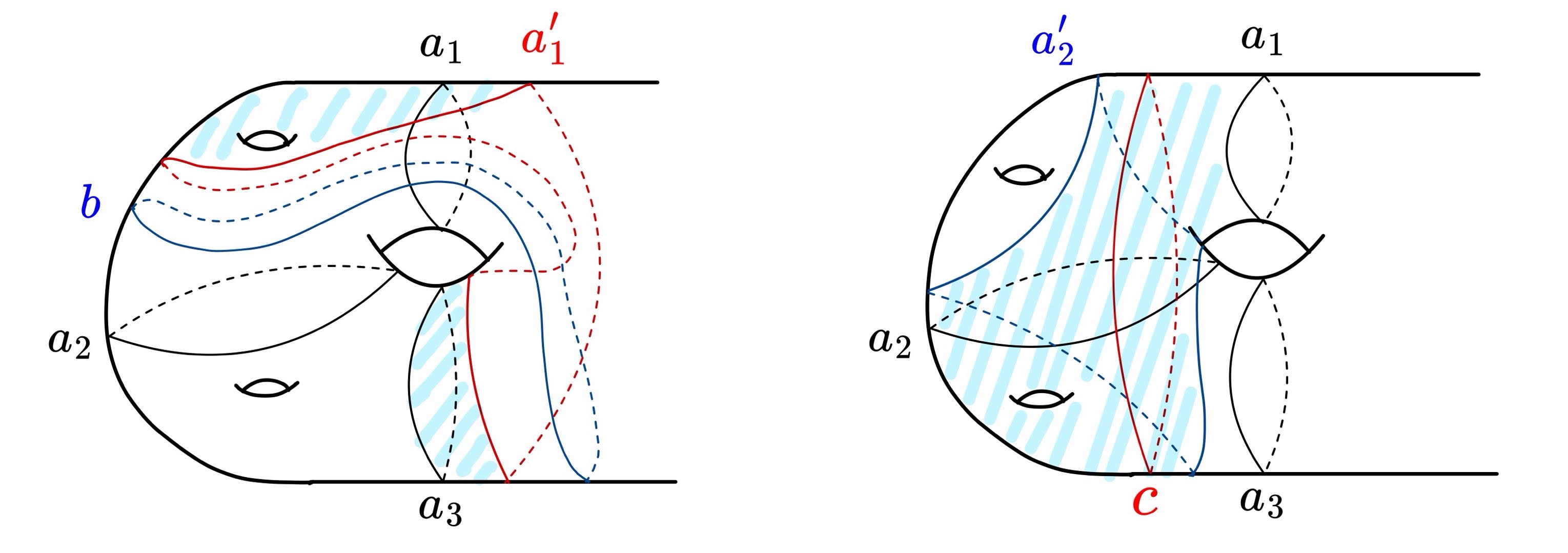}
    \caption{Left: the blue-shaded region is a genus 1 subsurface bounded by $(a_1',a_3)$. Right:  the blue-shaded region is a genus 1 subsurface bounded by $(a_1,a_2')$.}
    \label{fig:two lanterns}
\end{figure}

Consider the lantern formed by the curves $a_2$ and $c$ depicted on the right side of Figure \ref{fig:two lanterns} above. Let $a_2'$ denote the third diagonal curve in the lantern. The lantern relation gives that 
\begin{equation*}
    T_{a_2}T_{a_2'}T_{c} \equiv T_{a_1} T_{a_3} T_1^2\quad \Rightarrow \quad H_0:= T_{a_1}T_{a_3}T_{a_2}^{-2}\equiv (T_{a_2}^{-1}T_{a_2'}) T_1^{-2}T_2.
\end{equation*}
The two pairs $(a_1,a_2)$ and $(a_1,a_2')$ are both genus 1 bounding pairs. Hence, $a_2$ and $a_2'$ are homologous with $\g(a_2,a_2')=0$. By Lemma \ref{lem: i=2}, we have that 
$$T_{a_2}^{-1}T_{a_2'} \equiv B_0.$$
Moreover, we have that $B_0\equiv 0$ by Proposition \ref{2.2}. Hence, this lantern relation reduces to that 
\begin{equation}
    \label{eq: lantern 3}
    H_0 \equiv T_1^{-2}T_2.
\end{equation}

By  (\ref{eq: lantern 2}) and (\ref{eq: lantern 3}), we have that 
$$H_0^2 \equiv T_1T_3T_2^{-1}\equiv T_1^{-4}T_2^{2}.$$
This equation together with (\ref{eq: lantern 1}) implies that 
$T_1^2\equiv 0.$

Since $d$ is an $\M$-invariant homomorphism on $\Ch$ such that $d(T_n)=4n(n-1)$, by Proposition \cite[Proposition 5.1 and Theorem 5.3]{morita2} , we conclude that $d(H_0)=d(T_1^{-2}T_2)=d(T_2) = 8$ by (\ref{eq: lantern 3}).
\end{proof}

% This implies that 
% \begin{equation}
%     \label{eq: h^2}
%     (T_{a_1}T_{a_3}T_{a_2}^{-2})(T_{a_2}T_{c}T_{a_3}^{-2}) = T_{a_1}T_{c}T_{a_2}^{-1}T_{a_3}^{-1}
% \end{equation}

% \[B_0H_0\cong T_1^{-2}T_2\]
% Then we know that $h\cong T_1^{-2}T_2$.

% We have another lantern relation
% \[
% T_2T_{a_1}T_{a_4} \cong T_{a_2}T_{a_3}T_1T_3
% \]
% Then we have
% \[
% T_{a_1}T_{a_4}T_{a_2}^{-1}T_{a_3}^{-1} \cong T_1T_3T_2^{-1}
% \]
% On the other hand, we have
% \[T_{a_1}T_{a_3}T_{a_2}^{-2}T_{a_2}T_{a_4}T_{a_3}^{-2}\cong H_0^2\]

% Thus we know that 
% \[
% H_0^2\cong T_1T_3T_2^{-1}
% \]
% We thus obtain
% \[
% (T_1^{-2}T_2)^2 \cong (T_1^{-1}T_2)^2
% \]
% We obtain that $T_1^2\cong 0$.

\section{Calculations about $Ch_g^1$}

   % calculate $\Ch/[\Ch,\M]$. In particular, we will use results in \cite{Normal} to prove that $T_1 \in [\Ch,\M]$. As a consequence, we will prove Theorem \ref{main2}, Theorem \ref{thm:homological results}, and the first equality of Theorem \ref{main1}. 

    The main goal of this section is to prove that $T_1\in [\Ch,\M]$ when $g\ge 6$. Our proof is indirect and uses results in \cite{Normal} in an unexpected way. Once we finish proving that $T_1\in [\Ch,\M]$, we deduce several theorems stated in the introduction as consequences. Let us give a brief overview of this Section.    
   In 3.1, we study the algebraic intersection form as a cohomology class. In 3.2, we compute $H^1(\Ch;\Z/m\Z)^{\M}$ using results in \cite{Normal}. In 3.3, we prove that $T_1\in [\Ch,\M]$. In 3.4, we prove the right equality of Theorem \ref{main1}. In 3.5, we prove Theorem \ref{thm:homological results}.

% \subsection{Previous works}

% Morita constructed a crossed homomorphism $k: \M \to H=H_1(S_g^1;\Z)$ such that 
% \begin{equation}
%     \label{eq: k properties}
%     k(\phi\psi)=\psi^{-1}(k(\phi))+k(\psi)\ \ \ \ \text{ for any $\phi,\psi\in\M$}. 
% \end{equation}
% Moreover, he \cite[Theorem 6.1]{Jacobi2} proved that the restriction of $k$ of $\I$ is equal to the Chillingworth homomorphism $\ch$ up to a sign. 

% The \emph{Casson-Morita map} is a map $d:\I\to \Z$ satisfying the following properties. 
% \begin{proposition}[Morita, Proposition 5.1 and Theorem 5.3 in \cite{Morita2}]
% \label{d map properties}
% For any two elements $\phi\in \I,\ \psi \in \M$, 
% \begin{enumerate}
% \item 
% $d ( \phi \psi ) = d ( \phi ) + d ( \psi ) +\ii(k(\phi),\psi _ { * } k(\psi ))$;
% \item
% $d ( \phi ^ { - 1 } ) = - d ( \phi )$;
% \item
% $d ( \phi \psi \phi ^ { -1 } ) = d ( \psi ) + k( \phi ) \cdot( \psi _ { * } k( \psi ) + k( \phi \psi ) )$
% \end{enumerate}
% Moreover, if $T_h$ denotes a separating twist of genus $\alpha$, then $d(T_h)=4h(h-1).$
% \end{proposition}

% $d$ is not a group homomorphism on $\I$, but it follows from Proposition \ref{} that its restriction to is a group homomorphism. 

\subsection{The cohomology class of intersection form}
\label{sec:intersection form}
To analyze the group $\Ch$, we first study the homology of the abelian group $H:=H_1(S_g^1;\Z)$ and $H_{m}:=H_1(S_g^1;\Z/m)$ for an integer $m$. Let $\ii$ denote the algebraic intersection form
$$\ii: H\times H\to \Z.$$
Then $\ii$ is a cocyle and defines a cohomology class in $H^2(H;\Z)$.
% Note that $\ii$ can be regarded as an element in $H^2(H;\Z)$, under the natural maps:
% \begin{equation}
%     \label{eq:Hom H2 identified}
%     Hom(H^{\otimes2},\Z)\cong Hom(H,\Z)^{\otimes 2} \twoheadrightarrow \wedge^2 Hom(H,\Z)\cong \wedge^2 H^1(H;\Z)\xrightarrow{\smallsmile} H^2(H;\Z)
% \end{equation}

\begin{proposition}
\label{prop:i=2j}
For any $m\in \Z$, we have the following results. 
\begin{enumerate}
    \item There exists $j\in H^2(H;\Z)$ such that $\ii=2j$. Moreover, $j$ is $\M$-invariant and is a generator for
$$H^2(H;\Z)^{\M} \cong \Z.$$
\item Let $j'$ denote the image of $j$ under natural map:
\begin{align*}
    H^2(H;\Z) &\to H^2(H;\Z/m\Z)\\%\to H^2(H_{2m};\Z/m\Z)\\
    \qquad j&\mapsto j'%\qquad\longmapsto\qquad j''
\end{align*}
Then we have that 
$$H^2(H;\Z/m\Z)^{\M}=\langle j'\rangle \cong \Z/m\Z %\text{\ \ \ \  and \ \ \ \  } H^2(H_{2m};\Z/m\Z)^{\M}=\langle j''\rangle \cong \Z/m\Z
.$$
\item Let 
$$r: H^2(H_{2m};\Z/m)\to H^2(H;\Z/m)$$
denote the map induced by the projection $H\to H_{2m}$. Then there exists an element $j''\in H^2(H_{2m};\Z/m)$ such that $r(j'')=j'$.
\end{enumerate}
\end{proposition}

\begin{proof}
Let $\{a_l,b_l\}_{1\le l\le g}$ denote a symplectic basis for $H\cong \Z^{2g}$. We will consider the intersection pairing $\ii$ not just as a cohomology class, but also as an explicit 2-cochain $\ii:H\times H\to \Z$ with the following formula:
\[
\ii(\sum_{l=1}^g (\lambda_la_l+\mu_lb_l), \sum_{l=1}^g ( \lambda_l'a_l+\mu_l'b_l))=\sum_{l=1}^g (\lambda_l\mu_l' - \lambda_l'\mu_l).\]
We will modify $\ii$ by adding a coboundary.  Define $f:H\to \Z$ by 
\[
f(\sum_{l=1}^g  \lambda_la_l+\mu_lb_l)= \sum_{l=1}^g  \lambda_l\mu_l 
\]
We then have that
\begin{align*}
(\ii+\delta f)&\Big(\sum_{l=1}^g (\lambda_la_l+\mu_lb_l), \sum_{l=1}^g ( \lambda_l'a_l+\mu_i'b_l)\Big)\\
    =&\sum_{l=1}^g \lambda_l\mu_l' - \lambda_l'\mu_l + (\lambda_l+\lambda_l')(\mu_l+\mu_l')-\lambda_l\mu_l-\lambda_l'\mu_l' \\
    =&2 \sum_{l=1}^g \lambda_l\mu_l'\\    
\end{align*}
The computation above shows that $\ii+\delta f$ is always a multiple of $2$. Define \[
j:=\frac{\ii+\delta f}{2}.\]
In other word, we define a 2-cochain
\begin{equation}
    \label{eq: j}
    j\Big(\sum_{l=1}^g (\lambda_la_l+\mu_lb_l), \sum_{l=1}^g ( \lambda_l'a_l+\mu_l'b_l)\Big)=\sum_{l=1}^g \lambda_l\mu_l'
\end{equation}
Since $2j-\ii=\delta f$, we have that $2[j]=[\ii]$ as cohomology classes in $H^2(H;\Z)$. Since $[\ii]$ is $\M$-invariant, we have that $[j]$ is also $\M$-invariant because $H^2(H;\Z)$ is torsion-free. 

Next we show that $[j]$ is a generator for $H^2(H;\Z)^\M$.  
Consider the dual basis $\{a_k^*,b_k^*\}_{1\le k\le g}$ for $H^1(H,\Z)$. We have that 
$$(\wedge^2 H^1(H;\Z))^\M\cong \Z$$
where a generator is given by the class
$$\sum_{k=1}^g (a_k^*\wedge b_k^*).$$
The cup product on $H^*(H;\Z)$ gives an $\M$-equivariant isomorphism 
\begin{equation}
    \label{wedge 2 iso}
    \wedge^2 H^1(H;\Z)\xrightarrow{\smallsmile} H^2(H;\Z).
\end{equation}
This map takes the class $\sum_{k=1}^g (a_k^*\wedge b_k^*)$ to $[j]$ by (\ref{eq: j}). Hence, $[j]$  generates $H^2(H;\Z)^\M\cong \Z$. Part (1) is established. 

Part (2) follows from the universal coefficient theorem. Notice that $H^t(H;\Z)$ is torsion-free for all $t$. 

To prove (3), we want to find a class $j''\in H^2(H_{2m};\Z/m\Z)$ such that $r(j'')=j'$. 
% It suffices to show that $j$ descends as a map $k: H_{2m}\times H_{2m}\to \Z/m$. 
It follows from (\ref{eq: j}) that
\[
j(x+2mz,y)=j(x,y) \mod 2m,\qquad \ \ \forall x,y,z\in H.
\]
Thus $j:H\times H\to\Z$ descends to a well-defined map $j'':H_{2m}\times H_{2m}\to\Z/m$, which restricts to $j':H\times H\to\Z/m\Z$.
 
Finally, to show that $j''$ is invariant under the action of $\M$, we check it on the generators of $\M$. The Humphries generating set of $\M$ consists of  Dehn twists $T_{a_k},T_{b_k}$  for $k\in \{1,2,...,g\}$ and $T_{a_1+a_2}$. We calculate:
\begin{align*}
(j^{T_{a_k}}-j)&(\sum_{l=1}^g (\lambda_la_l+\mu_lb_l), \sum_{l=1}^g ( \lambda_l'a_l+\mu_i'b_l)) \\
=&j(T_{a_k}(\sum_{l=1}^g (\lambda_la_l+\mu_lb_l)), T_{a_k}(\sum_{l=1}^g ( \lambda_l'a_l+\mu_i'b_l)))\\&-j(\sum_{l=1}^g (\lambda_la_l+\mu_lb_l), \sum_{l=1}^g ( \lambda_l'a_l+\mu_i'b_l))\\
=&j(\sum_{l=1}^g (\lambda_la_l+\mu_lb_l)-\mu_ka_k, \sum_{l=1}^g ( \lambda_l'a_l+\mu_i'b_l)-\mu_k'a_k) \\&-j(\sum_{l=1}^g (\lambda_la_l+\mu_lb_l), \sum_{l=1}^g ( \lambda_l'a_l+\mu_i'b_l))\\
=&\sum_{l=1}^g \lambda_l\mu_l' -\lambda_k\mu_k'+(\lambda_k-\mu_k)\mu_k'-\sum_{l=1}^g \lambda_l\mu_l' \\
=&-\mu_k\mu_k'
\end{align*}
Define $q_1^k:H_{2m}\to \Z/m$ by
\[
q_1^k(\sum_{l=1}^g (\lambda_la_l+\mu_lb_l))=\mu_k(\mu_k+1)/2   \pmod m.
\]
Observe that $q_1^k$ only depends on $\mu_k \pmod {2m}$ and hence is well-defined.
\begin{align*}
\delta &q_1^k(\sum_{l=1}^g (\lambda_la_l+\mu_lb_l), \sum_{l=1}^g ( \lambda_l'a_l+\mu_i'b_l))   \\
=&q_1^k(\sum_{l=1}^g  (\lambda_l+\lambda_l')a_l+(\mu_l+\mu_l')b_l)-q_1^k(\sum_{l=1}^g  \lambda_la_l+\mu_lb_l)-q_1^k(\sum_{l=1}^g  \lambda_l'a_l+\mu_l'b_l))
 \\
=&(\mu_k+\mu_k')(\mu_k+\mu_k'-1)/2-\mu_k(\mu_k+1)/2+\mu_k'(\mu_k'+1)/2\\
=&\mu_k\mu_k' 
\end{align*}
We thus have $\delta q_1^k=-(j^{T_{a_k}}-j)$ and hence $[j^{T_{a_k}}]=[j]$. We remark that from here, we can see that $j \pmod {2m}$ may not be $\M$-invariant. At least our proof does not work for this stronger claim because the cochain $q_1^k$ does not decend to a well-defined map $H_{2m}\to\Z/2m\Z$.

We now discuss another generator $T_{a_1+a_2}$, we have

\begin{align*}
(j^{T_{a_1+a_2}}-j)&(\sum_{l=1}^g (\lambda_la_l+\mu_lb_l), \sum_{l=1}^g ( \lambda_l'a_l+\mu_l'b_l)) \\
=&j(T_{a_1+a_2}(\sum_{l=1}^g (\lambda_la_l+\mu_lb_l)), T_{a_1+a_2}(\sum_{l=1}^g ( \lambda_l'a_l+\mu_l'b_l)))\\&-j(\sum_{l=1}^g (\lambda_la_l+\mu_lb_l), \sum_{l=1}^g ( \lambda_l'a_l+\mu_l'b_l))\\
=&j(\sum_{l=1}^g (\lambda_la_l+\mu_lb_l)-(\mu_1+\mu_2)(a_1+a_2), \sum_{l=1}^g ( \lambda_l'a_l+\mu_l'b_l)-(\mu_1'+\mu_2')(a_1+a_2))\\&-j(\sum_{l=1}^g (\lambda_la_l+\mu_lb_l), \sum_{l=1}^g ( \lambda_l'a_l+\mu_l'b_l))\\
=&\sum_{l=1}^g \lambda_l\mu_l' -\lambda_1\mu_1'-\lambda_2\mu_2'+(\lambda_1-\mu_1-\mu_2)\mu_1'+(\lambda_2-\mu_1-\mu_2)\mu_2'-\sum_{l=1}^g \lambda_l\mu_l' \\
=&-(\mu_1+\mu_2)(\mu_1'+\mu_2')
\end{align*}

Define $q_2:H_{2m}\to \Z/m$ by
\[
q_2(\sum_{l=1}^g \lambda_la_l+\mu_lb_l)=\mu_1\mu_2
\]
We calculate that 
\begin{equation}
\begin{aligned}
\delta &q_2(\sum_l  \lambda_la_l+\mu_lb_l, \sum_l  \lambda_l'a_l+\mu_l'b_l))   \\
=&q_2(\sum_l  (\lambda_l+\lambda_l')a_l+(\mu_l+\mu_l')b_l)-q(\sum_l  \lambda_la_l+\mu_lb_l)-q(\sum_l  \lambda_l'a_l+\mu_l'b_l))
 \\
=&(\mu_1+\mu_1')(\mu_2+\mu_2')-\mu_1\mu_2-\mu_1'\mu_2' = \mu_1\mu_2'+\mu_1'\mu_2\\
\end{aligned}
\end{equation}
Thus we have
\[
\delta(q_1^1+q_1^2+q_2)=-(j^{T_{a_1+a_2}}-j)
\]
Our calculations above show that $[j]$ is $\M$-invariant.
\end{proof}

The Chillingworth homomorphism 
$\ch:\I\to H$
has $\ker(\ch) = \Ch$ and $\im(\ch) = 2H.$ Therefore, we have a short exact sequence of groups
\begin{equation}
    \label{eq: SES Ch I H}
    0\to \Ch\to\I\xrightarrow{\ch/2} H\to 0
\end{equation}
which induces the following five-term exact sequence:
\begin{equation}
    \label{eq: five over Z}
    0\to H^1(H;\Z)\to H^1(\I;\Z)\to H^1(\Ch;\Z)^H\xrightarrow{
    \delta} H^2(H;\Z)\to H^2(\I;\Z).
\end{equation}
Let $d:\Ch\to\Z$ denote the restriction of the Casson-Morita's $d$ map to the Chillingworth subgroup. By \cite[Proposition 19 and Proposition 20]{Chilling}]), its restriction to $\Ch$ is a $\M$-invariant group homomorphism with image $d(\Ch)=8\Z$. Thus $d/8$ is an $\M$-invariant homomorphism on $\Ch$. We have the following.
\begin{proposition}
    \label{prop: delta(d)}
    The connecting homomorphism $\delta:H^1(\Ch;\Z)^H\to H^2(H;\Z)$ in (\ref{eq: five over Z}) maps $d/8$ to $j$. 
\end{proposition}
\begin{proof}
In general, for $A$ an abelian group and for a short exact sequence of groups
\[
0\to N\to G\to Q\to 0,\]
the connecting homomorphism $\delta: H^1(N;A)^Q\to H^2(Q;A)$ can be defined as the following: Let $f: N\to A$ be a $Q$-invariant group homomorphism. Take any extension $f': G\to A$ of $f$ to $G$. Define $\delta (f):Q\times Q\to A$ by
\[\delta(f)(q_1,q_2)=f'(\tilde{q}_1\tilde{q}_2)-f'(\tilde{q}_1)-f'(\tilde{q}_2)\]  
where $\tilde{q}_i$ is any lift of $q_i$ to $G$. Then $\delta(f)$ is a 2-cocyle representing the cohomology class $\delta[f]$.  See \emph{e.g.} Proposition (1.6.6) and Theorem (2.4.3) in \cite{neukirch2013cohomology} for a proof of this general fact. 

Let us apply this general result to short exact sequence (\ref{eq: SES Ch I H}). Recall that the homomorphism $d:\Ch\to \Z$ is a restriction of a map $d:\I\to\Z$ such that for any two elements $\phi,\psi \in \I$, we have by Morita \cite[Proposition 5.1]{morita2} that
$$d ( \phi \psi ) = d ( \phi ) + d ( \psi ) +\ii(\ch(\phi),\ch(\psi )).$$
Take any elements $c_1,c_2\in H$. Let $\tilde{c}_i\in \I$ be their lifts such that $e(\tilde{c}_i)=2c_i$ for $i=1,2$. Then we have that 
$$\delta(d)(c_1,c_2) = d (\tilde{c}_1\tilde{c}_2)-d (\tilde{c}_1) - d (\tilde{c}_2)  = \ii(\ch(\tilde{c}_1),\ch(\tilde{c}_2)) = 4\ii(c_1,c_2).$$
Therefore, we have that 
$\delta(d)=4\ii = 8j$ as elements in $H^2(H;\Z)$ where $j$ is the generator specified Proposition \ref{prop:i=2j}. Since $H^2(H;\Z)$ is torsion free, we have that 
$$\delta(d/8) = j.$$
\end{proof}

\subsection{\boldmath The computation of $H^1(\Ch;\Z/m\Z)^{\M}$}
\begin{theorem}
\label{thm:Abmodm}
Let $m$ be a positive integer such that $2m\leq g-2$. We have that
\[H^1(\Ch;\Z/m\Z)^{\M}\cong \Z/m\Z,
\]
which is generated by $d/8 \pmod m$.
\end{theorem}
\begin{proof}
\noindent\textbf{Step 1: relate $\Ch$ and $\mathbf{\Ch[4m]}$ by exact sequences.}
Consider the Chillingworth congruence subgroup 
$$\Ch[4m]:=\{f\in\I\ : \ \ch(f)=0\pmod{4m}\}.$$
Then 
$\Ch$ and $\Ch[4m]$ fit into a commutative diagram of short exact sequences where the first row is exactly (\ref{eq: SES Ch I H}) above:
\begin{center}
% https://q.uiver.app/#q=WzAsMTIsWzAsMCwiMCJdLFsxLDAsIlxcQ2giXSxbMiwwLCJcXEkiXSxbMywwXSxbNCwwLCJIIl0sWzAsMSwiMCJdLFsxLDEsIlxcQ2hbMm1dIl0sWzIsMSwiXFxJIl0sWzMsMV0sWzQsMSwiSF97Mm19Il0sWzUsMCwiMCJdLFs1LDEsIjAiXSxbMCwxXSxbMSwyXSxbNSw2XSxbNiw3XSxbMSw2LCIiLDEseyJzdHlsZSI6eyJ0YWlsIjp7Im5hbWUiOiJob29rIiwic2lkZSI6InRvcCJ9fX1dLFsyLDcsIj0iXSxbNyw5LCJcXGNoXFxwbW9kIDJtIl0sWzIsNCwiXFxjaCJdLFs0LDksIiIsMSx7InN0eWxlIjp7ImhlYWQiOnsibmFtZSI6ImVwaSJ9fX1dLFs0LDEwXSxbOSwxMV1d
\begin{tikzcd}
	0 & \Ch & \I & {} & H & 0 \\
	0 & {\Ch[4m]} & \I & {} & {H_{2m}} & 0
	\arrow[from=1-1, to=1-2]
	\arrow[from=1-2, to=1-3]
	\arrow[hook, from=1-2, to=2-2]
	\arrow["\ch/2", from=1-3, to=1-5]
	\arrow["{=}", from=1-3, to=2-3]
	\arrow[from=1-5, to=1-6]
	\arrow[two heads, from=1-5, to=2-5]
	\arrow[from=2-1, to=2-2]
	\arrow[from=2-2, to=2-3]
	\arrow["{\ch/2\pmod {2m}}", from=2-3, to=2-5]
	\arrow[from=2-5, to=2-6]
\end{tikzcd}
\end{center}
The two horizontal short exact sequences of groups induce two five term long exact sequences, which fit into the following commutative diagram. Here the cohomology are taken with coefficients in $\Z/m\Z$ which we suppress from notation.
% \begin{center}
%     % https://q.uiver.app/#q=WzAsMTIsWzAsMCwiMCJdLFsxLDAsIkheMShIO1xcWi9tKSJdLFsyLDAsIkheMShcXEksXFxaL20pIl0sWzMsMCwiSF4xKFxcQ2g7XFxaL20pXkgiXSxbNCwwLCJIXjIoSCxcXFovbSkiXSxbNSwwLCJIXjIoXFxJLFxcWi9tKSJdLFsxLDEsIkheMShIX3sybX07XFxaL20pIl0sWzIsMSwiSF4xKFxcSSxcXFovbSkiXSxbMywxLCJIXjEoXFxDaFsybV07XFxaL20pXkgiXSxbNCwxLCJIXjIoSF97Mm19LFxcWi9tKSJdLFs1LDEsIkheMihcXEksXFxaL20pIl0sWzAsMSwiMCJdLFswLDFdLFsxLDJdLFsyLDNdLFszLDQsIlxcZGVsdGEiXSxbNCw1XSxbNiw3XSxbNyw4XSxbOCw5XSxbOSwxMF0sWzExLDZdLFs2LDEsIlxcY29uZyJdLFs3LDIsIj0iXSxbOCwzLCJSIl0sWzksNCwiciJdLFsxMCw1LCI9Il1d
% \begin{tikzcd}
% 	0 & {H^1(H;\Z/m)} & {H^1(\I,\Z/m)} & {H^1(\Ch;\Z/m)^H} & {H^2(H,\Z/m)} & {H^2(\I,\Z/m)} \\
% 	0 & {H^1(H_{2m};\Z/m)} & {H^1(\I,\Z/m)} & {H^1(\Ch[2m];\Z/m)^H} & {H^2(H_{2m},\Z/m)} & {H^2(\I,\Z/m)}
% 	\arrow[from=1-1, to=1-2]
% 	\arrow[from=1-2, to=1-3]
% 	\arrow[from=1-3, to=1-4]
% 	\arrow["\delta", from=1-4, to=1-5]
% 	\arrow[from=1-5, to=1-6]
% 	\arrow[from=2-1, to=2-2]
% 	\arrow["\cong", from=2-2, to=1-2]
% 	\arrow[from=2-2, to=2-3]
% 	\arrow["{=}", from=2-3, to=1-3]
% 	\arrow[from=2-3, to=2-4]
% 	\arrow["R", from=2-4, to=1-4]
% 	\arrow[from=2-4, to=2-5]
% 	\arrow["r", from=2-5, to=1-5]
% 	\arrow[from=2-5, to=2-6]
% 	\arrow["{=}", from=2-6, to=1-6]
% \end{tikzcd}
% \end{center}
\begin{equation}
\xymatrix{
H^1(H_{2m}) \ar[r]^{\  (e/2)^*}\ar[d]^\cong & H^1(\I)  \ar[r]^{}\ar[d]^{=}& H^1(\Ch[4m])^{H_{2m}}  \ar[r]^{\ \ \ \ \delta''}\ar[d]^R&  H^2(H_{2m})   \ar[r]\ar[d]^r & H^2(\I) \ar[d]^{=}  \\
H^1(H) \ar[r]^{(e/2)^*} & H^1(\I)  \ar[r]^{}& H^1(\Ch)^{H}  \ar[r]^{\delta^\prime}& H^2(H) \ar[r] & H^2(\I) }              
\label{eq:diagram1}
\end{equation}
The leftmost vertical map is an isomorphism because any homomorphism $H\to \Z/m\Z$ necessarily factors through a map $H_{2m}=H/2mH\to\Z/m\Z$. Define $P:=\mathrm{coker} (e/2^*)$. Then the long exact sequences in (\ref{eq:diagram1}) become the following short exact sequences: (the coefficients $\Z/m\Z$ are still suppressed from notation)
\begin{equation}
\label{eq:diag2}
\xymatrix{
0\ar[r] & P  \ar[r]^{}\ar[d]^{=}& H^1(\Ch[4m])^{H_{2m}}  \ar[r]^{\ \ \ \ \ \  \delta''}\ar[d]^R&  \im(\delta'') \ar[r] \ar[d]^r & 0  \\
0\ar[r] & P  \ar[r]^{}& H^1(\Ch)^{H}  \ar[r]^{\delta'} & \im(\delta') \ar[r] & 0
}              
\end{equation}
All modules in the diagram above have natural actions of the mapping class group $\M$ such that all maps are $\M$-equivariant. Therefore, by taking $\M$-invariants of (\ref{eq:diag2}), we have the following long exact sequences:
\begin{equation}
\label{eq:diag3}
\xymatrix{
0\ar[r] & P^\M  \ar[r]^{}\ar[d]^{=}& H^1(\Ch[4m])^{\M}  \ar[r]^{\ \ \ \ \delta''}\ar[d]^R&  \im(\delta'')^\M \ar[r] \ar[d]^r & H^1(\M;P)  \ar[d]^=\\
0\ar[r] & P^\M  \ar[r]^{}& H^1(\Ch)^{\M}  \ar[r]^{\delta'} & \im(\delta')^\M \ar[r] & H^1(\M;P)
}              
\end{equation}

\noindent\textbf{Step 2: analyze the exact sequence (\ref{eq:diag3}).}
Next, we will study the exact sequence (\ref{eq:diag3}). Let us first recall some known results which we will use.

Recall that the Casson-Morita's $d$ homomorphism is a map $d:\I\to \Z$ such that for any two elements $\phi,\psi \in \I$, 
\begin{equation}
    \label{eq: d properties}
    d ( \phi \psi ) = d ( \phi ) + d ( \psi ) +\ii(\ch(\phi),\ch(\psi )).
\end{equation}
Please see Morita \cite{morita2}, Proposition 5.1 for more discussion. Define $\phi:\Ch[4m]\to\Z/m\Z$ by
    $$\phi(f):=d(f)/8\pmod {m}.$$
%Let us briefly recall some properties of $d$ that will be used later in our proof. 
When $1\le 2m\le g-2$, we have the following results in \cite{Normal}:
\begin{enumerate}
    \item \cite[Proposition 4.10]{Normal} $d(\Ch)=d(\Ch[4m])=8\Z$.
    \item \cite[Lemma 4.2]{Normal} 
    The homomorphism $\phi$  is a $\M$-invariant group homomorphism.
    \item \cite[Theorem 1.4]{Normal} The homomorphism $\phi$ induces an isomorphism 
    $$H_1(\Ch[4m];\Z)_{\M}\cong \Z/m\Z.$$
\end{enumerate}
By the universal coefficient theorem and (3) above, we have that 
$$H^1(\Ch[4m];\Z/m\Z)^\M =\langle\phi\rangle \cong \Z/m\Z. $$
Hence, the proof of Theorem \ref{thm:Abmodm} will be complete if we can show that the vertical map $R$ in (\ref{eq:diag3}) is an isomorphism, which we will prove in the following claim.

\begin{claim}
    \label{lem: sequence 10}
    We have the following results about the exact sequence (\ref{eq:diag3}).
    \begin{enumerate}
        \item $\delta'(R(\phi))=j'$ where $j'$ is the generator for $H^2(H;\Z/m\Z)^\M$ defined in Proposition \ref{prop:i=2j}. 
        % \item $\im(\delta')^\M = H^2(H;\Z/m\Z)^\M$.
        \item The vertical maps $r$ in (\ref{eq:diag3}) is surjective.
        \item The vertical map $R$ in (\ref{eq:diag3}) is an isomorphism.
    \end{enumerate}
\end{claim}
\begin{proof}
Recall that by Proposition \ref{prop: delta(d)}, the connecting homomorphism
$$\delta: H^1(\Ch;\Z)^H\to H^2(H;\Z)$$
maps $d/8$ to $j$. In our case now, if we reduce the coefficients by $\Z\to \Z/m\Z$, then the connecting homomorphism
$$\delta': H^1(\Ch;\Z/m\Z)^H\to H^2(H;\Z/m\Z)$$
maps $d/8\pmod{m}$ to $j'=j \pmod m$. Finally, notice that  $d/8\pmod{m}$ is equal to $R(\phi)$, which is the restriction of $\phi$ to the subgroup $\Ch$. Hence we obtain (1).

Since the diagram (\ref{eq:diag3}) commutes, we have that
$$\delta'(R(\phi))=r(\delta''(\phi))=j'.$$
    Since $j'$ is a generator for $H^2(H;\Z/m\Z)^\M$ by Proposition \ref{prop:i=2j} and is in the image of $r$, we have that  $r$ in  (\ref{eq:diag3}) must be surjective. By the Snake Lemma, $R$ in  (\ref{eq:diag3}) is also surjective. 
     It remains to show that $R$ in  (\ref{eq:diag3}) is also injective.
     
    Since $R(\phi)=d/8 \pmod m$ and that $d/8:\Ch\to \Z$ is surjective, we obtain that $d/8 \pmod m: \Ch\to \Z/m$ is surjective. Therefore we have that $R(k\phi)\neq 0$ if $k$ is not a multiple of $m$. 
    %It remains to show that $R$ in  (\ref{eq:diag3 }) is also injective. Suppose that $R(k\phi)=0$ for some $k\in\Z/m\Z$. Then the restriction $k\phi:\Ch\to\Z/m\Z$ is the zero map. However, since $d(\Ch)=8\Z$, we know that there exists some element $f\in \Ch$ such that $d(f)=8$. In fact, we can take $f$ to be any genus 2 separating twist. Then we must have that $k\phi(f)=k=0\pmod m$. Hence $R$ is injective.    
    \end{proof}
    The proof of Theorem \ref{thm:Abmodm} is complete.
\end{proof}

\subsection{\boldmath The proof of that $T_1\in [\Ch,\M]$.}

\begin{proposition}
\label{prop: T1 in commutator}
For $g\ge 6$, we have that 
\[{T_1\in [\Ch,\M].}\]
\end{proposition}
\begin{proof}
Let $A:=\Ch/[\Ch,\M]$. Since $\Ch$ is normally generated by $B_0,T_1,T_2$, we know that $A$ is generated by $B_0,T_1,T_2$. Since $B_0, T_1^2\in [\Ch,\M]$, we know that $A$ as an abelian group is generated by $T_1,T_2$ and that $T_1$ is a 2-torsion. Since we have a homomorphism $d:A\to\Z$ such that $d(T_2)=8\ne 0$ by \cite[Theorem 5.3]{morita2}. We have that $T_2$ must be of infinite order in $A$. Hence, by the classification of finitely generated abelian group, if $T_1\notin [\Ch,\M]$, we have that $A\cong \Z\oplus \Z/2$. However, this implies that
\[
H^1(\Ch;\Z/2\Z)^{\M}\cong\text{Hom}(A;\Z/2) \cong \Z/2\oplus \Z/2
\]
This contradicts Theorem \ref{thm:Abmodm} when $m=2$ where we need $g-2\ge 4$. We thus obtain that $T_1\in [\Ch,\M]$ when $g\ge 6$.
\end{proof}

\subsection{The proof of the right equality of Theorem \ref{main1}}
\begin{proof}
Since $d:\Ch\to\Z$ is a $\M$-invariant group homomorphism, we automatically have that $[\Ch,\M]\leq \ker d$. We just need to prove the reverse inclusion. Kosuge \cite{Chilling} showed that $\ker d$ is normally generated by $B_0,T_1$, and $[\J,\M]$. It is clear that $[\J,\M]\le [\Ch,\M]$. We proved that $T_1\in [\Ch,\M]$ in Proposition \ref{prop: T1 in commutator} when $g\ge 6$. We have $B_0\in  [\Ch,\M]$ by Proposition \ref{2.2}. It follows that $\ker d= [\Ch,\M]$.
\end{proof}

\subsection{The proof of Theorem \ref{thm:homological results}}
First of all, we state a simple lemma. 
\begin{lemma}
\label{lem:H1 and commutator}
    Let $G$ be a normal subgroup of a group $M$. Suppose that $f:G\to A$ is an $M$-invariant surjective homomorphism onto an abelian group $A$ such that 
    $\ker f\subseteq[G,M].$
    Then we have that 
    \begin{align}
        H_1(G)_M&\cong A \ \\
        H^1(G;B)^M&\cong Hom(A,B) \ \text{ for any abelian group $B$.}
    \end{align}
\end{lemma}
\begin{proof}
Since $f$ is $M$-invariant, we have that $[G,M]=\ker f$. 
\begin{align*}
    &H_1(G)_M \cong G/[G,M] = G/\ker f\cong A,\\
    &H^1(G;B)^M\cong Hom(G,B)^M \cong Hom(H_1(G)_M,B){\cong} Hom(A,B).\qedhere
\end{align*}
    \end{proof}
We now start the proof of Theorem \ref{thm:homological results}.
\begin{proof}
By Lemma \ref{2.3}, we have that $\W(0)= [\W(0),\M]$, which implies part (1) of Theorem \ref{thm:homological results} by Lemma \ref{lem:H1 and commutator}.

Part (2) of Theorem \ref{thm:homological results} follows from Lemma \ref{lem:H1 and commutator} and first equation of Theorem \ref{main1}.
\end{proof}
\vspace*{4ex}

\section{The curve complex $C_0^\alpha$ and the proof of Theorem \ref{C_0}}
The main goal of this Section is to prove Theorem \ref{C_0}. Using this result, we will deduce the rest of the main theorems that we stated in the Introduction. Let us give a brief overview of this Section. In 4.1, we prove that various curve complexes are connected.  Using those results,  in 4.2, we prove that the complex $C_0^\alpha$ is connected.  In 4.3, we show that the actions of $\Ch$ and $\W(0)$ on $C_0^\alpha$ are both transitive. Hence, we finish proving the two parts of Theorem \ref{C_0}. Finally, we prove Theorem \ref{main2} which states that $H_0$ normally generates $\Ch$. 

Throughout this section, we will fix  a primitive element $\alpha$ in $H:= H_1(S_g^1;\Z)$. For two curves  $b,c$ both homologous to $\alpha$, we will denote their homological genus 
$g^\alpha(b,c)$ simply by $g(b,c)$, suppressing $\alpha$ from the notation. 

\begin{defn} 
$C_0^\alpha$ is a curve complex of $S_g^1$ consisting of vertices and edges described below:
\begin{itemize}
\item vertex: a curve $c$ such that $g(\alpha,c)=0$. 
\item edge: two vertices $c,d\in C_0^a$ form an edge if and only if $T_cT_d^{-1}$ is conjugate to $B_0$ in $\M$.
\end{itemize}
\end{defn}
% In this section, we will prove two results: one is that $C_0^\alpha$ is path-connected; the other is that the actions of $\Ch$ on $C_0^\alpha$ is transitive. 

\subsection{Putman's connectivity and its applications}
Putman proved the following useful lemma in  \cite{Putman}.
\begin{lemma}[Putman's connectivity lemma]\label{Putman}
Consider a group $G$ acting upon a simplicial complex $X$ where $X_n$ denotes its $n$th skeleton. Fix a basepoint $v \in X_0$ and a set $S$ of generators for $G$. Assume the following hold.
\begin{enumerate}
\item For all $v' \in X_0$, the orbit $Gv$ intersects the connected component of $X$ containing $v'$;
\item  For all $s \in S^{\pm 1}$, there is some path $P(s)$ in $X$ from $v$ to $sv$.
\end{enumerate}
Then X is connected.
\end{lemma}

We consider the following sub-complex $C(a,p,q)$ of $C(a)$.
\begin{defn} For $a\in C(\alpha)$ and $p,q\in\Z$,  define $C(a,p,q)$ as the following
\begin{itemize}
\item vertex: a curve $c\in C(a)$ disjoint from $a$ such that $g(a,c)=p$ or $g(a,c)=q$
\item edge: two vertices $b,c\in C(a,p,q)$ form an edge if $i(b,c)=0$.
\end{itemize}
\end{defn}
We now prove the following connectivity lemma.
\begin{lemma}\label{connectivity1}
For $g\ge 4$, we have the following.
\begin{enumerate}
\item 
When $1\le p<q \le g-1$, the complex $C(a,p,q)$ is path-connected.
\item
The complex $C(a,1,-1)$ is path-connected.
\end{enumerate}
\end{lemma}
\begin{proof}
Let us consider the compact surface $S_g^1(a)$ with boundary such that $S_g^1(a)$ is obtained by adding boundary components to each end in $S_g^1-a$. The mapping class group $\mathcal{M}(S_g^1(a))$ acts on $C(a,p,q)$ since $\mathcal{M}(S_g^1(a))$ preserves the genera of sub-surfaces.

Let $v=c$ be a curve in $C(a,p,q)$ such that $g(a,c)=p$. By the change of coordinate principle, we know that the action of $\M(S(a))$ on $C(a,p,q)$ is transitive on curves $c$ such that $g(a,c)=p$. Let $d\in C(a,p,q)$ be such that $g(a,c)=q$. Since $p<q\le g-1$, there is a curve $c'\in C(a,p,q)$ disjoint from $d$ such that $g(a,c')=p$. Thus $d$ is connected by a path in $C(a,p,q)$ to $\mathcal{M}(S_g^1(a))\cdot v$. 

The mapping class group $\M(S(a))$ is generated by the Dehn twists about black curves in Figure \ref{s31}. See e.g. \cite[Chapter 5]{primer} for more discussion.  We call this generating set $S$.
\begin{figure}[h]
        \centering
        \includegraphics[width=1\linewidth]{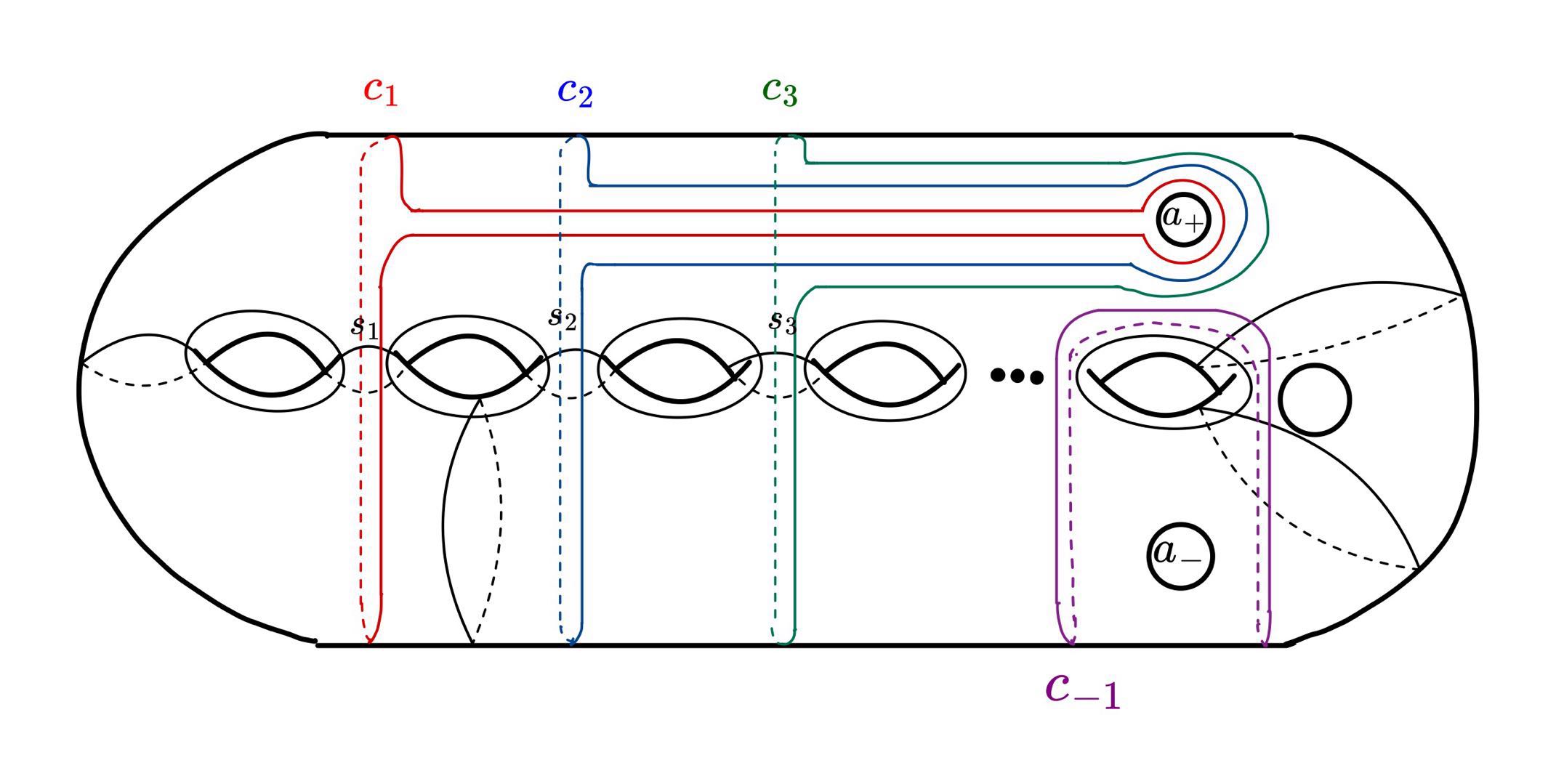}
        \vspace{-20pt}
        \caption{Generating set of $\mathcal{M}(S_g^1(a))$} 
        \label{s31}
    \end{figure}

 We now choose $v:=c_p$ as in Figure \ref{s31}. The only element in the generating set $S$ that intersect $c_p$ is $s_p$. For $p\neq q$, we also know that $T_{s_{p}}(c_p)$ is disjoint from $s_q$. By Lemma \ref{Putman},  we know that $C(a,p,q)$ is path-connected.

For (2), let $v:=c_1$. By the change of coordinate principle, the action of $\mathcal{M}(S_g^1(a))$ on $C(a,1,-1)$ is transitive on curves $c$ such that $g(a,c)=1$. Let $d\in C(a,p,q)$ be a curve such that $g(a,c)=-1$. Then there is a curve $c'\in C(a,1,-1)$ such that $g(a,c')=1$ disjoint from $d$. Thus $d$ is connected by a path to $Gv$. 

Now the only elements in the generating set $S$ that intersect $v$ is $s_1$. We also know that $T_{s_1}(c_1)$ is disjoint from $c_{-1} \in C(a,1,-1)$. Hence, $c_1$ is path-connected to $T_{s_1}(c_1)$ via $c_{-1}$. By Lemma \ref{Putman}, we know that $C(a,1,-1)$ is path-connected.
\end{proof}

%here!!!
\begin{defn} We define a curve complex $C^{\text{disj}}(a,-1)$ as the following.
\begin{itemize}
\item vertex: a curve $c$ such that $c$ is homologous to $a$ and is disjoint from $a$ and $g(a,c)=-1$.
\item edge: two vertices $c,d\in C^{\text{disj}}(a,-1)$ form an edge if there exists a simple closed curve $a'$ homologous to $a$ and disjoint from $c,d$ such that $g(a,a')=1$.
\end{itemize}
\end{defn}
\begin{lemma}\label{connectivity2}
The complex $C^{\text{disj}}(a,-1)$ is path-connected.
\end{lemma}
\begin{proof}
Since $C(a,1,-1)$ is path-connected from Lemma \ref{connectivity1}, by the change of orientation on the curve $a$, we know that $C(a,-1,1)$ is path-connected as well. 

Let $c,d\in C^{\text{disj}}(a,-1)$, there is a path $c_1=c,c_2,...,c_n=d$ in $C(a,-1,1)$ connecting $c,d$. Since a pair of curves $x,y$ such that $g(a,x)=g(a,y)$ must intersect, we know that $g(a,c_{2i-1})=-1$ and $g(a,c_{2i})=1$. Since $g(a,c_{2i})=1$ and that $i(c_{2i},c_{2i-1})=i(c_{2i},c_{2i+1})=0$, we know that $c_1,c_3,...,c_n$ is a path in $C^{\text{disj}}(a,-1)$.
\end{proof}

Let $b$ be a simple closed curve such that $g(\alpha,b)=2$. We now consider the following curve complex a curve complex $C_0^\alpha(b)$ which is a sub-complex of $C_0^\alpha$.
\begin{defn} We define a curve complex $C_0^\alpha(b)$ as the following.
\begin{itemize}
\item vertex: a curve $c\in C_0^\alpha$ such that $i(b,c)=0$;
\item edge: two vertices $c,d\in C_0^\alpha(b)$ form an edge if and only if $T_cT_d^{-1}$ is conjugate to $B_0$ in $\M$.
\end{itemize}
\end{defn}
We now prove the following.
\begin{lemma}\label{connectivity3}
Let $b$ be a simple closed curve such that $g(\alpha,b)=2$. Then $C_0^\alpha(b)$ is path-connected.
\end{lemma}
\begin{proof}
Let us consider the compact surface $S_g^1(b)$ with boundary such that $S_g^1(b)$ is obtained by adding boundary components to each end in $S_g^1-b$. The mapping class group $\mathcal{M}(S(b))$ acts on $C(b,p,q)$ since $\mathcal{M}(S_g^1(b))$ preserves the genera of sub-surfaces. 

Let $v:=c_2$. By the change of coordinate principle, we know that the action of $\mathcal{M}(S_g^1(b))$ on $C_0^\alpha(b)$ is transitive. By \cite[Chapter 5]{primer}, the mapping class group $\mathcal{M}(S_g^1(b))$ is generated by the Dehn twists about white curves in Figure \ref{s32}.
\begin{figure}[h]
        \centering
        \includegraphics[width=1\linewidth]{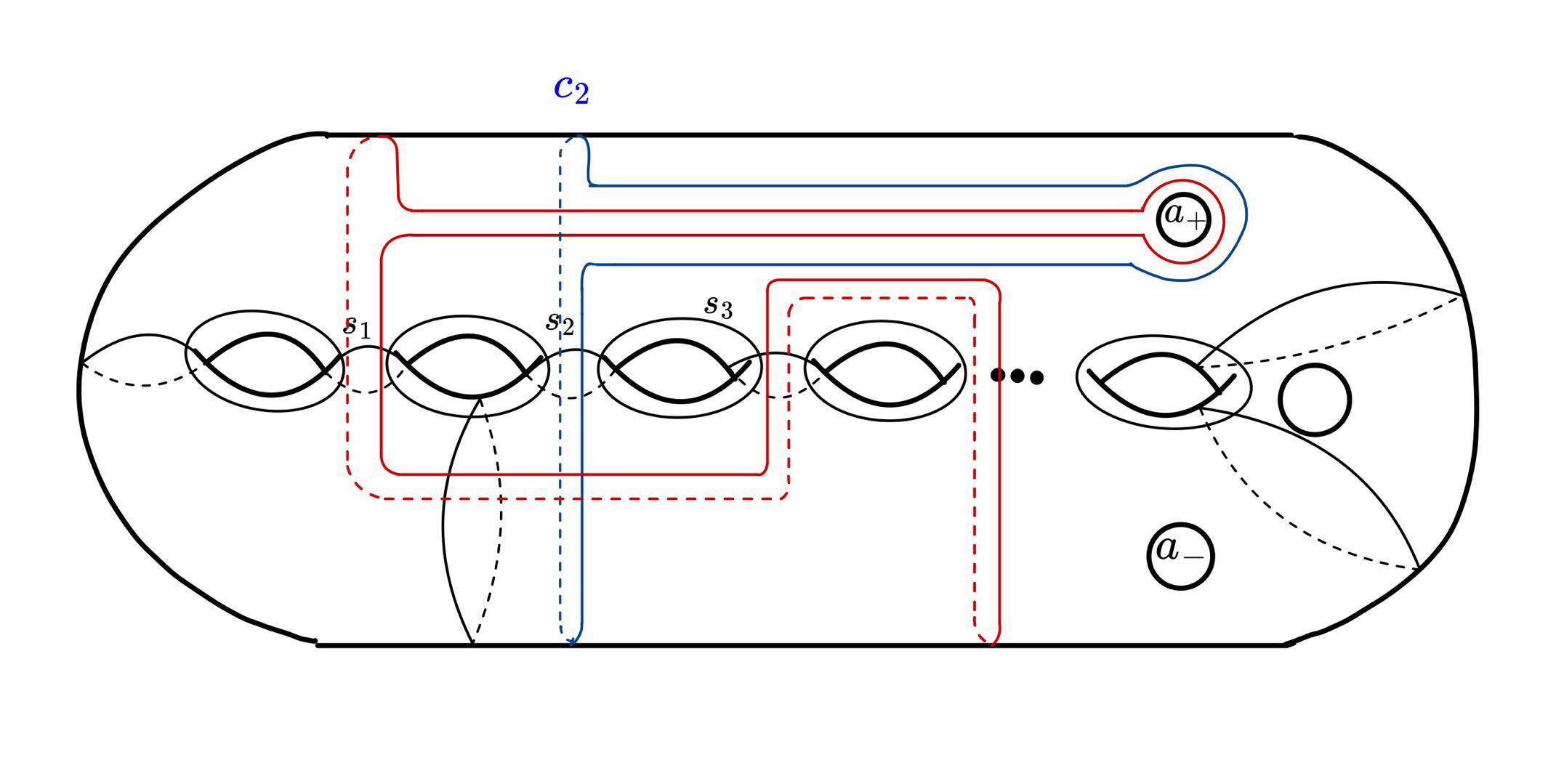}
        \vspace{-20pt}
        \caption{Generating set of $\M(S(a))$} 
        \label{s32}
    \end{figure}
Now the only elements in the generating set $S$ that intersect $c_2$ is $s_2$. Consider the red curve $r$ in the Figure \ref{s32}. Observe that $i(r,v)=2$. Since $r$ is disjoint from $s_2$, \[
i(r,T_{s_{3}}(v))=i(r,v)=2\]. Since $g(b,r)=g(b,v)=2$, we know \[
g(r,T_{s_{2}}(v))=g(r,v)=g(r,b)+g(b,v)=-2+2=0.\] 

By Lemma \ref{lem: i=2}, we know that $v,r$ is an edge in $C_0^\alpha(b)$ and $v,T_{s_{2}}(v)$ is an edge in $C_0^\alpha(b)$. We know that $v,T_{s_{2}}(v)$ is connected by a path of length $2$ in $C_0^\alpha(b)$.

By Lemma \ref{Putman},  we then know that $C_0^\alpha(b)$ is path-connected.
\end{proof}

\subsection{\boldmath The proof that $C_0^\alpha$ is path-connected}
The goal of this section is to prove the part (1) of Theorem \ref{C_0} which claims that $C_0^\alpha$ is connected. 

Let $a,a'$ be two curves in $C_0^\alpha$. We will prove that there is a path between $a,a'$ in $C_0^\alpha$. We break the proof into three steps.

\p{\textbf{\boldmath Step 1: We can find a path $P_1:=a_0=a,a_1,...,a_n=a'$ in $C(\alpha)$ such that $0\le g(\alpha,a_i)\le 1$}}

\begin{proof}
Firstly, we find a path $P$  in $C(\alpha)$
\[
a_0=a,a_1,...,a_n=a'
\]
connecting $a,a'$. Now $a_i,a_{i+1}$ are disjoint homologous curves such that $a_i,a_{i+1}$ form a bounding pair. If $P$ does not satisfying the condition, we prove by modifying the path to $\min_{x\in P}\{ g(\alpha,x)\}$. Suppose that $\min_{x\in P}\{ g(\alpha,x)\}<0$,

Let $j$ be the index such that $g(\alpha,a_j) = \min\{ g(\alpha,a_i)\}$ and by the assumption, we know $g(\alpha,a_j)<0$. Then since $a_{j-1},a_{j+1}$ are both disjoint from $a_j$, we know that $g(\alpha,a_{j-1})\neq g(\alpha,a_j)$ and $g(\alpha,a_{j+1})\neq g(\alpha,a_j)$. We  then have that 
\[  g(a_j,a_{j-1})= g(a_{j},\alpha)+ g(\alpha,a_{j-1}) = -g(\alpha, a_{j})+ g(\alpha,a_{j-1})>0\]
and 
\[
 g(a_j,a_{j+1})= g(a_{j},\alpha)+ g(\alpha,a_{j+1})= -g(\alpha, a_{j})+ g(\alpha,a_{j+1}) >0.
\]
This implies that $a_{j-1}$ and $a_{j+1}$ are on the same side of $a_j$.

Let $p=g(a_j,a_{j-1})$ and $q=g(a_j,a_{j+1})$. 

\begin{itemize}
\item
If $p\neq q$, let $C=C(a_j,p,q)$;
\item
If $p=q\neq 1$, let $C=C(a_j,1,q)$;
\item 
If $p=q=1$, let $C=C(a_j,1,2)$
\end{itemize}

Then since $C$ is path-connected, we know that there is a path between $a_{j-1}$ and $a_{j+1}$ in $C$. Replace the $a_{j-1},a_j,a_{j+1}$ part of $P_1$ by this new path in $C$, we delete one minimum point in $\min_{x\in P'}\{ g(\alpha,x)\}$ and do not increase $\max_{x\in P_1}\{ g(\alpha,x)\}$. We can continue this procedure until $\min_{x\in P'}\{ g(\alpha,x)\} = 0$. We call the new path $P_1'$ where $\min_{x\in P_1'}\{ g(\alpha,x)\} = 0$

The proof that we can modify the path $P_1'$ with $P_1''$ such that $\max_{x\in P_1''}\{ g(\alpha,x)\}=1$ is the same. This proves that we can find the path as in Step 1.
\end{proof}

\p{\textbf{\boldmath Step 2: We can find a path $P_2: a_0=a,a_1,\cdots,a_n=a'$ in $C(\alpha)$ such that $g(\alpha,a_i)$ is either $0$ or $2$}}
\begin{proof}
Since $i(a_i,a_{i+1})=0$, we know that $g(a_i,a_{i+1})\neq 0$. We thus know that $g(a,a_{2i})=0$ and $g(a,a_{2i+1})=1$ for allowable $i$. 

Consider the curve $a_{2i+1}$, since the complex $C^{disj}(a_{2i+1},-1)$ is connected by Lemma \ref{connectivity2}, we know that $a_{2i}$ and $a_{2i+2}$ can be connected by a path $c_0,...,c_k$ in $C^{disj}(a_{2i+1},-1)$. We thus know that $g(\alpha,c_j)=0$. Each pair $c_j,c_{j+1}$, there exists $d_j$ disjoint from both $c_j,c_{j+1}$ and that $g(a_{2i+1},d_j)=1$. Thus
\[
g(\alpha,d_j)=g(\alpha,a_{2i+1})+g(a_{2i+1},d_j)=2.
\]
This shows that $a_{2i},a_{2i+2}$ is connected by a path in $C(a)$ whose vertices $v$ satisfying that either $g(\alpha,v)=0$ or $g(\alpha,v)=2$. We thus obtain a path $P_2$ connecting $a,a'$ satisfying Step 2.
\end{proof}
\p{\textbf{\boldmath Step 3: There is a path $P_3$ in $C_0^a$ connecting $a,a'$}}
We now prove (1) of Theorem \ref{C_0}.
\begin{proof}
Now $P_2$ is a path in $C(\alpha)$ such that $g(\alpha,a_i)$ is either $0$ or $2$. Since $i(a_i,a_{i+1})=0$, we know that $g(a_i,a_{i+1})\neq 0$. We thus know that $g(\alpha,a_{2i})=0$ and $g(\alpha,a_{2i+1})=2$. By Lemma \ref{connectivity3}, we have that $C_0^\alpha(a_{2i+1})$ is path-connected. We can thus find a path in $C_0^a(a_{2i+1})$ between $a_{2i},a_{2i+2}$, which is also a path in  $C_0^a$. This concludes the proof.
\end{proof}

\subsection{Transitivity}

% \begin{proposition}
% The action of $\langle \langle B_0 \rangle\rangle$ on $C_0^\alpha$ is transitive.
% \end{proposition}
We now prove (2) of Theorem \ref{C_0}. 
\begin{proof}[Proof of Theorem \ref{C_0}, part (2)]
By Lemma \ref{2.3}, we know that if $c,d$ is an edge in $C_0^\alpha$, there exists $h\in \W(0)$ such that $h(c)=d$. Since we also know that $C_0^\alpha$ is path-connected by section 4.2, for any two vertices $c,d\in C_0^\alpha$, there is a path $c_0=c,...,c_n=d$ in $C_0^\alpha$. Then there exists $h_i\in \W(0)$ such that $h_i(c_{i-1})=c_i$. Then we know that
\[
h_nh_{n-1}...h_1(c_0)=c_n.
\]
This implies that the action of $\W(0)$ on $C_0^\alpha$ is transitive. Since $\Ch$ contains $\W(0)$, the action of $\Ch$ on $C_0^\alpha$ is also transitive.
\end{proof}
We can now finish proving the left equality of Theorem \ref{main1}.
\begin{proof}[The proof of the left equality of Theorem \ref{main1}]
Firstly, by Lemma \ref{2.2}, we know that $\W(0)\leq [\Ch,\M]$. We now prove the reverse inclusion. Since $\M$ is generated by Dehn twists along nonseparating curves, we know that $[\Ch,\M]$ is generated by $fT_cf^{-1}T_c^{-1}$ for $f\in \Ch$ and nonseparating curves $c$. Then the curve $f(c)$ is a vertex in $C_0^c$ (The definition of $C_0^c$ does not depend on the orientation of $c$). By Theorem \ref{C_0}, there exists $h\in \langle \langle B_0 \rangle\rangle$ such that $h(c)=f(c)$. We then have that 
\[
fT_cf^{-1}T_c^{-1}=T_{f(c)}T_c^{-1} = T_{h(c)}T_c^{-1}=hT_ch^{-1}T_c^{-1}\in \langle \langle B_0 \rangle\rangle
\]
This implies that $[\Ch,\M]\leq\W(0)$.
\end{proof}
We now finish the proof Theorem \ref{main2}.
\begin{proof}[Proof of Theorem \ref{main2}]
Consider the Casson--Morita's $d$-map $d:\Ch\to\Z$. It is enough to show $d(H_0)$ generates $\im(d)$ and that 
$\langle\langle H_0\rangle\rangle$ contains $\ker (d)$.

By Lemma \ref{imd}, we have that $d(H_0)=d(T_2)$. By Morita \cite[Theorem 5.2]{morita2}, we have that $d(T_2)=8$. On the other hand, Kosuge \cite[Proposition 20]{Chilling} showed that $\im(d)=8\Z$. Hence, $d(H_0)$ generates $\im(d)$. Moreover, by Theorem \ref{main1} and Proposition \ref{2.2}, we have that $\ker(d) = \langle\langle B_0\rangle\rangle\leq \langle\langle H_0\rangle\rangle$. Hence, $H_0$ normally generates $\Ch$.
\end{proof}

% \begin{theorem}
% \label{main1}
% When $1\le n\le g-2$, we have that
% \begin{enumerate}
% \vspace{3pt}
%     \item $\W(n)=[\Ch[2n],\M] = [\W(n),\M]$
%     \vspace{3pt}
%     \item 
% \end{enumerate}
% \end{theorem}

%By \cite[Theorem 1.2]{Normal}, we know that $Ch_g^1[2n]/\W(n) = \Z/n$ if $n$ odd and $\Z/\frac{n}{2}$ if $n$ even and that $\W(n)=[Ch_g^1[2n],\M]$. These two statement imply that $H_1(Ch_g^1[2n])_{\M}$ is $\Z/n$ if $n$ odd and $\Z/\frac{n}{2}$ if $n$ even. We now only consider the case when $n$ is even.

%Let $n=2m$.  Dually, we know
%\[
%H^1(Ch_g^1[2m];\Z/m)^{\M} =  \Z/m,\]
%this implies that

\bibliographystyle{plain}
\bibliography{BPsep}
\end{document}